\documentclass[11pt]{article}

\usepackage[margin=1in]{geometry}
\usepackage[T1]{fontenc}
\usepackage{lmodern}
\usepackage{amsmath,amssymb,amsthm,mathtools}
\usepackage{booktabs,tabularx,array}
\usepackage{microtype}
\usepackage{parskip}
\usepackage{xcolor}
\usepackage[
  colorlinks=true,
  linkcolor=blue!55!black,
  citecolor=blue!55!black,
  urlcolor=blue!65!black,
  pdftitle={Codimension-one simplex configurations and pinned simplices in Salem sets},
  pdfsubject={Falconer-type simplex configurations},
  pdfkeywords={simplices, Hausdorff dimension, Salem sets, pinned configurations, Frostman measures}
]{hyperref}

\newtheorem{theorem}{Theorem}[section]
\newtheorem{proposition}[theorem]{Proposition}
\newtheorem{lemma}[theorem]{Lemma}
\newtheorem{corollary}[theorem]{Corollary}
\theoremstyle{remark}
\newtheorem{remark}[theorem]{Remark}

\newcommand{\R}{\mathbb R}
\newcommand{\Sph}{\mathbb S}
\newcommand{\cL}{\mathcal L}
\newcommand{\cH}{\mathcal H}
\newcommand{\supp}{\operatorname{supp}}
\newcommand{\dist}{\operatorname{dist}}

\newcommand{\eps}{\varepsilon}
\newcommand{\wh}{\widehat}
\newcolumntype{Y}{>{\raggedright\arraybackslash}X}
\numberwithin{equation}{section}

\newcommand{\Fq}{\mathbb F_q}

\newcommand{\cT}{\mathcal T}
\newcommand{\cY}{\mathcal Y}
\newcommand{\cA}{\mathcal A}

\newcommand{\Orth}{\operatorname O}
\newcommand{\Span}{\operatorname{span}}
\newcommand{\rank}{\operatorname{rank}}
\newcommand{\rad}{\operatorname{rad}}
\newcommand{\1}{\mathbf 1}
\DeclareMathOperator{\Aff}{Aff}

\newcolumntype{P}[1]{>{\raggedright\arraybackslash}p{#1}}

\title{Distribution of simplices in the discrete and continuous settings}
\author{Thang Pham\and Chun-Yen Shen\and Boqing Xue}
\date{}

\begin{document}
\maketitle
\begin{abstract}
In this paper, we study the distribution of simplices in both discrete and continuous settings. Let \(q\) be an odd prime power, let \(Q\) be a nondegenerate
quadratic form on \(\mathbb F_q^d\), and let \(2\leq k\leq d-1\).
We prove that every set \(E\subset\mathbb F_q^d\) with
\[
 |E|\geq C_{d,k}q^{\beta_{d,k}},
 \qquad
 \beta_{d,k}=
 \begin{cases}
 \displaystyle \frac{d+k}{2}-\frac{k-1}{k+1},
     & d-k\ \text{even},\\[2mm]
 \displaystyle \frac{d+k-1}{2},
     & d-k\ \text{odd},
 \end{cases}
\]
determines a positive proportion of all ordered nondegenerate
\(k\)-simplex congruence classes. This improves the previous exponent due to Bennett,
Hart, Iosevich, Pakianathan, and Rudnev (2017), and is sharp when \(d-k\) is odd. In the Euclidean setting, we prove that if \(E\subset\mathbb R^d\) is
compact and \(\dim_{\mathrm H}(E)>d-1\), then there exists a Frostman
probability measure \(\mu\), supported on \(E\), and a set of pins of
full \(\mu\)-measure such that the pinned distance
configuration measure for labeled \((d-1)\)-simplices is absolutely
continuous at every such pin. We also show that the same conclusion holds when \(E\subset\mathbb R^d\) is a compact Salem set with
\(\dim_{\mathrm H}(E)>k\). 
\end{abstract}


\section{Introduction}
In this paper, we build on the algebraic framework introduced by the authors \cite{PSX} for distances and triangles and extend it to the study
of higher-dimensional simplices.  Our work may therefore be viewed as
a continuation and further development of \cite{PSX}: the underlying
algebraic philosophy remains central, while additional geometric and
inductive ingredients are introduced to handle simplices of arbitrary
dimension.

\subsection{Results over finite fields}
For \(k\geq1\), we write $t_k=\binom{k+1}{2}$. Let \(q\) be an odd prime power, let \(V=\Fq^d\), and let
\(Q:V\to\Fq\) be a nondegenerate quadratic form.  We write
\[
 B(u,v)=\frac{Q(u+v)-Q(u)-Q(v)}2
\]
for the associated symmetric bilinear form and denote the orthogonal
group of \((V,Q)\) by \(\Orth(Q)\).

An ordered \(k\)-simplex in \(E\subset V\) is a tuple
\[
        \mathbf x=(x_0,x_1,\ldots,x_k)\in E^{k+1}.
\]
Two ordered simplices \(\mathbf x\) and \(\mathbf x'\) are
\emph{congruent} if
\[
        x_i'=z+\theta x_i\qquad(0\leq i\leq k)
\]
for some \(z\in V\) and \(\theta\in\Orth(Q)\).  With
\(u_i=x_i-x_0\), the edge Gram matrix of \(\mathbf x\) is
\[
        G(\mathbf x)
        =
        \bigl(B(u_i,u_j)\bigr)_{1\leq i,j\leq k}.
\]
We call \(\mathbf x\) \emph{nondegenerate} if
\(\det G(\mathbf x)\neq0\).  By Witt's extension theorem, two
nondegenerate ordered simplices are congruent if and only if their edge
Gram matrices are equal.

Let \(\cT^d_{k,Q}(E)\) denote the set of ordered \(k\)-simplex
congruence classes determined by \(E\), and let
\(\cT^{d,\mathrm{nd}}_{k,Q}(E)\) denote its subset of nondegenerate
classes.  For \(2\leq k\leq d-1\), the nondegenerate classes in \(V\)
are in one-to-one correspondence with the nonsingular symmetric
\(k\times k\) matrices.  Consequently,
\[
 \bigl|\cT^{d,\mathrm{nd}}_{k,Q}(V)\bigr|
 =
 q^{t_k}+O_k(q^{t_k-1}).
\]
Thus a positive proportion result has the form
\[
 \bigl|\cT^{d,\mathrm{nd}}_{k,Q}(E)\bigr|
 \gg q^{t_k}.
\]
All implicit constants below may depend on \(d\) and \(k\), but not
on \(q\) or \(E\).

The finite field simplex problem has been studied using Fourier
analysis, orthogonal-group actions, and incidence estimates, see \cite{BHIPR, 
BIP, CEHIK, McDonald, Parshall, PhamFixedSide, VinhBilinear, VinhKaleidoscopic}. The record was due to Bennett,
Hart, Iosevich, Pakianathan, and Rudnev in \cite{BHIPR}. They proved that, for
every nondegenerate \(Q\),
\begin{equation}\label{eq:BHIPR-intro}
 |E|\gg q^{\,d-(d-1)/(k+1)}
 \quad\Longrightarrow\quad
 |\cT^d_{k,Q}(E)|\gg q^{t_k}.
\end{equation}

For \(2\leq k\leq d-1\), define
\begin{equation}\label{eq:beta}
 \beta_{d,k}=
 \begin{cases}
 \displaystyle \frac{d+k}{2}-\frac{k-1}{k+1},
       & d-k\ \mathrm{even},\\[2mm]
 \displaystyle \frac{d+k-1}{2},
       & d-k\ \mathrm{odd}.
 \end{cases}
\end{equation}
Our main finite field result reads as follows.
\begin{theorem}\label{thm:finite-main}
Let \(d\geq3\), let \(q\) be an odd prime power, let \(Q\) be a
nondegenerate quadratic form on \(\Fq^d\), and let
\(2\leq k\leq d-1\).  There are constants
\(C_{d,k},c_{d,k}>0\) such that every set \(E\subset\Fq^d\)
satisfying
\[
 |E|\geq C_{d,k}q^{\beta_{d,k}}
\]
also satisfies
\[
 |\cT^{d,\mathrm{nd}}_{k,Q}(E)|
 \geq c_{d,k}q^{t_k}.
\]
If \(d-k\) is odd, then the exponent $\beta_{d,k}=(d+k-1)/2$ is optimal.
\end{theorem}

Theorem~\ref{thm:finite-main} improves the exponent in
\eqref{eq:BHIPR-intro} throughout the range \(2\leq k\leq d-1\).
More precisely, the gain is
\[
 \left(d-\frac{d-1}{k+1}\right)-\beta_{d,k}
 =
 \begin{cases}
 \displaystyle
 \frac{k-1}{k+1}\,\frac{d-k}{2},
       & d-k\ \mathrm{even},\\[3mm]
 \displaystyle
 \frac{(k-1)(d-k-1)/2+1}{k+1},
       & d-k\ \mathrm{odd}.
 \end{cases}
\]
Both quantities are positive. Sharpness in the odd-codimension case is established in
Section~\ref{sec:sharpness}.

Note that if $k=2$ and $d$ is even, then it has been proved in \cite{PSX} that the better threshold of $q^{\frac{d}{2}+\frac{3}{5}}$ is sufficient. 

\subsection{Results in the Euclidean setting}

If \(E\subset\R^d\) is compact, its labelled \(k\)-simplex distance
set is
\[
\Delta_k(E)
=
\left\{
\bigl(\lvert x_i-x_j\rvert\bigr)_{0\leq i<j\leq k}
:
x_0,\ldots,x_k\in E
\right\}
\subset\R^{t_k}.
\]
Here \emph{labelled} means that no quotient by permutations of the
vertices is taken.  We write \(\Delta_k^{\mathrm{nd}}(E)\) for the
subset arising from affinely independent vertices.  The Falconer
simplex problem asks for dimensional hypotheses on \(E\) that imply
\[
        \cL^{t_k}\bigl(\Delta_k(E)\bigr)>0.
\]
For \(k=1\), this is the classical Falconer distance problem \cite{Falconer}.

Fix \(x\in E\), put \(z_0=x\), and, for \(1\leq j<d\), define the
pinned squared-distance map
\begin{equation}
\Phi_{j,x}(z_1,\ldots,z_j)
=
\bigl(\lvert z_a-z_b\rvert^2\bigr)_{0\leq a<b\leq j}
\in\R^{t_j}.
\label{eq:pinned-map-intro}
\end{equation}
For $1\le j<d$, $x\in\R^d$, and $z_0=x$, define the pinned
ordinary-distance map
\[
D_{j,x}(z_1,\ldots,z_j)
=
\bigl(
|z_a-z_b|
\bigr)_{0\le a<b\le j}
\in\R^{t_j}.
\]
Thus, for a probability measure $\mu$, the pinned squared-distance and
ordinary-distance configuration measures are
\[
(\Phi_{j,x})_\#\mu^j
\qquad\text{and}\qquad
(D_{j,x})_\#\mu^j,
\]
respectively, where $\mu^j$ denotes the $j$-fold product measure.

We write \(\Delta_{j,x}(E)\) for the corresponding pinned
ordinary-distance set and \(\Delta_{j,x}^{\mathrm{nd}}(E)\) for the
subset arising from affinely independent vertices.

Our first Euclidean result is on the case $k=d-1$.

\begin{theorem}\label{thm:euclidean-codim-one}
Let \(d\geq3\), let \(m=\binom d2\), and let \(E\subset\R^d\) be
compact.  If
\[
        \dim_{\mathrm H}(E)>d-1,
\]
then there exist a probability measure \(\mu\), supported on \(E\),
and a Borel set \(E_\mu\subset\supp\mu\), with \(\mu(E_\mu)=1\), such
that, for every \(x\in E_\mu\),
\[
        (\Phi_{d-1,x})_\#\mu^{d-1}\ll\cL^m.
\]
Consequently,
\[
        \cL^m\bigl(\Delta_{d-1,x}^{\mathrm{nd}}(E)\bigr)>0
        \qquad (x\in E_\mu).
\]
\end{theorem}

Several approaches to Euclidean simplex configurations have been
developed, including Fourier-analytic, group-action, and microlocal
methods
\cite{EHI,GGIP,GILP,GILPElementary,GIT,GITMicrolocal, IPPS}.

For general compact sets, the best previously known unpinned result in
the codimension-one case follows by combining the group-action
criterion of Greenleaf--Iosevich--Liu--Palsson \cite{GILP} with the
spherical-average estimate of Du--Zhang \cite{DuZhang}; see also
\cite[Theorem~1.16]{BOP}.  It gives
\[
 \dim_{\mathrm H}(E)>d-1+\frac1{d+1},
\]
under which the natural \((d-1)\)-simplex configuration measure has an
\(L^2\) density. For the pinned problem, the best previous result is due to
Borges--Ou--Pasquariello \cite[Corollary~1.18]{BOP}, using the
graph-building framework of Borges et al.\
\cite[Theorem~B]{BFOPR}.  They proved that
\[
 \dim_{\mathrm H}(E)>
 d-1+\frac{d}{2d+1}
\]
ensures the existence of a point \(x\in E\) which pins a
positive-measure family of \((d-1)\)-simplices.

Thus Theorem~\ref{thm:euclidean-codim-one} lowers the previous pinned
codimension-one threshold from
\(d-1+\frac{d}{2d+1}\) to \(d-1\), and strengthens the conclusion from
one pin and positive measure to a full-\(\mu\)-measure set of pins and
absolute continuity. 

While the finite field analog (Theorem \ref{thm:finite-main}) is sharp, we do not have any examples to show that the dimensional threshold of $d-1$ in Theorem \ref{thm:euclidean-codim-one} is optimal. It follows from \cite[Theorem 1.8]{GILP} that the dimensional threshold cannot be smaller than $\min\{d-2, \frac{d}{2}\}$. If we assume further that $E$ is a Salem set, then the next result presents an improvement for $2\le k\le d-2$.

For a compact set \(E\subset\R^d\), its Fourier dimension is the
supremum of the exponents \(\beta\in[0,d]\) for which \(E\) supports a
probability measure \(\mu\) satisfying
\[
        \lvert\wh\mu(\xi)\rvert
        \lesssim
        (1+\lvert\xi\rvert)^{-\beta/2}.
\]
The set \(E\) is called a Salem set if its Fourier dimension is equal
to its Hausdorff dimension.  Equivalently, for every
\(0<\sigma<\dim_{\mathrm H}(E)\), the set \(E\) supports a probability
measure \(\mu\) satisfying
\[
        \lvert\wh\mu(\xi)\rvert
        \lesssim_\sigma
        (1+\lvert\xi\rvert)^{-\sigma/2}.
\]

Our second Euclidean result gives a pinned threshold which is
independent of the ambient dimension.

\begin{theorem}\label{thm:pinned-salem}
Let \(d\geq3\), let \(2\leq k\leq d-1\), and let
\(E\subset\R^d\) be a compact Salem set.  If
\[
        \dim_{\mathrm H}(E)>k,
\]
then there exist a probability measure \(\mu\), supported on \(E\),
and a Borel set \(E_\mu\subset\supp\mu\), with \(\mu(E_\mu)=1\), such
that, for every \(x\in E_\mu\),
\[
        (\Phi_{k,x})_\#\mu^k\ll\cL^{t_k}.
\]
Consequently,
\[
\cL^{t_k}\bigl(\Delta_{k,x}^{\mathrm{nd}}(E)\bigr)>0
\qquad (x\in E_\mu).
\]
\end{theorem}

By coordinate marginalization, the same measure $\mu$ and the same set
$E_\mu$ also satisfy
\[
(\Phi_{j,x})_\#\mu^j\ll\cL^{t_j}
\qquad
(1\le j\le k,\ x\in E_\mu).
\]
When \(k=d-1\), this conclusion already follows from
Theorem~\ref{thm:euclidean-codim-one} for arbitrary compact sets.
Thus the Sale specific improvement concerns the intermediate range
\(2\leq k\leq d-2\). The theorem does not cover the case $k=1$, and the current record for this case can be found in\cite[Theorem~2.4(1)]{FraserPham} due to Fraser and
Pham.

\subsection{Main ideas}
\label{sec:intro-main-ideas}

Bennett, Hart, Iosevich, Pakianathan, and Rudnev \cite{BHIPR}
used the group-action framework to reduce the \(L^2\) problem for
\(k\)-simplex multiplicities to an \(L^2\) distance-energy estimate.
The relevant scalar quantity is
\[
 \mathcal E_Q(E)
 :=
 \bigl|\{(x,y,z,w)\in E^4:
 Q(x-y)=Q(z-w)\}\bigr|.
\]
By Cauchy--Schwarz,
\[
 \mathcal E_Q(E)\geq \frac{|E|^4}{q},
\]
while the well-known estimate is
\[
 \mathcal E_Q(E)
 \leq
 \frac{|E|^4}{q}+Cq^d|E|^2,
\]
for some large constant $C$.

Thus, to improve the bound in \cite{BHIPR} within the same framework, it is
essential to retain the exact coefficient \(1\) in front of
\(|E|^4/q\) and improve the second term in the relevant large-set
range.  Indeed, a larger coefficient in the first term leaves a
main-sized contribution after centering.  Such an improvement appears
very difficult: even the conjecturally sharp spherical restriction
estimate in even dimensions improves the corresponding \(L^2\)
estimate only when
\[
 |E|<q^{d/2+1},
\]
and hence gives no gain at the codimension-one scale
\(|E|\sim q^{d-1}\) for even \(d\geq4\).  This difficulty motivates
our base--apex decomposition, in which the method in \cite{BHIPR} is applied
only to the base simplices and the apex energy is controlled separately
by point--hyperplane variance.

More precisely, the finite field argument adapts the extraction principle of the authors in \cite[Theorem~1.1]{PSX}, used there for distances and
triangles, to nondegenerate \(k\)-simplices in the range
\(2\leq k\leq d-1\).  Their theorem extracts
from a set in an even-dimensional quadratic space a large set in a
nondegenerate quadratic plane by quotienting by a large totally
isotropic subspace, while preserving every pairwise quadratic value
on the extracted set.  This quotient retains enough rank to transfer
planar triangle estimates, but a direct reduction to a plane cannot
produce a nondegenerate \(k\)-simplex when \(k\geq3\), because its edge Gram
matrix would have rank at most two.  We therefore quotient one
isotropic line at a time.  The main difficulty is to prevent too much
of \(E\) from collapsing in the quotient.  The estimate for isotropic
differences shows that one can choose an isotropic line \(\ell\) and
a coset of \(\ell^\perp\) whose image in the nondegenerate quotient
\(\ell^\perp/\ell\) remains large; the quotient lowers the ambient
dimension by two and preserves the entire edge Gram matrix.  Iteration
ends in dimension \(k\) when \(d-k\) is even and in dimension \(k+1\)
when \(d-k\) is odd.  The even-codimension case then uses the
full-dimensional estimate of Bennett et al.\ \cite{BHIPR}, whereas the
odd-codimension case uses the new codimension-one estimate.  To prove
that estimate, we write a \((d-1)\)-simplex as a
\((d-2)\)-simplex base and an apex.  An orthogonal-group moment bound
controls the second moment of the base-class multiplicities.  If two
distinct apices have the same distance vector, an isotropic difference
forces degeneracy; otherwise the base lies in their nondegenerate
perpendicular-bisector hyperplane, and the point--hyperplane variance
identity controls the resulting second moment.  Cauchy--Schwarz and
H\"older then give a
positive proportion of nondegenerate classes.  These two terminal
cases produce the two exponents in Theorem~\ref{thm:finite-main}.  The
odd-codimension exponent is sharp, while the gap in even codimension
is inherited from the full-dimensional terminal estimate. 

The passage to the Euclidean setting is conceptual rather than a
formal transfer between fields.  We retain the finite field
base--apex architecture: control the previously constructed face,
control the distances from a new vertex to that face, and then combine
the two pieces.  Multiplicity functions and finite sums are replaced
by pushforward measures and their densities; second-moment counts are
replaced by \(L^2\) estimates; and Cauchy--Schwarz assembly is replaced
by Fubini and measure disintegration.  The same perpendicular-bisector
linearization remains visible: equality of two squared distances is a
linear condition in the anchor, and the resulting collision kernel is
estimated through projections and Riesz energies.  This translation
also allows one pin to be kept fixed throughout the construction,
leading naturally to pinned absolute continuity rather than only an
unpinned positive-measure conclusion.

For Theorem~\ref{thm:euclidean-codim-one}, choose an \(s\)-Frostman
measure \(\mu\) with \(s>d-1\).  The affine
Blaschke--Petkantschin formula, together with an inverse-volume
estimate, shows that the distribution of the affine span of a
\(\mu^{d-1}\)-typical \((d-2)\)-simplex base is absolutely continuous
with respect to invariant measure on \(A(d,d-2)\).  For
\(L=a+V\), the auxiliary map
\[
        C_L(x)=\bigl(P_Vx,\lvert P_{V^\perp}x-a\rvert^2\bigr)
\]
records the tangential position of the apex and its squared normal
distance from \(L\).  Averaging the \(L^2\) norms of
\((C_L)_\#\mu\) over \(L\) produces the kernel
\(\lvert x-x'\rvert^{-(d-1)}\): the Grassmannian average contributes
\(\lvert x-x'\rvert^{-(d-2)}\), and integration in the offset
contributes one further power.  This is controlled by
\(I_{d-1}(\mu)\).  A global change of variables then converts the
auxiliary map into the full squared-distance star from the apex to
the base, giving an \(L^2\) star measure for almost every full-rank
anchor tuple.  Coordinate marginals provide the corresponding star
statement at every lower rank.  Fubini selects a common
full-\(\mu\)-measure set of pins, and a fixed-pin induction adds the
remaining vertices one at a time while retaining all previously
constructed edge data.  The complete pinned squared-distance measure
is therefore absolutely continuous; the ordinary-distance conclusion
and the exclusion of degenerate simplices follow from the
coordinatewise squaring map and the Gram-determinant polynomial.

The proof of Theorem~\ref{thm:pinned-salem} bypasses the affine-span
construction.  Choose
\(k<\sigma<\dim_{\mathrm H}(E)\) and a Salem measure \(\mu\) with
\[
        \lvert\wh\mu(\xi)\rvert
        \lesssim
        (1+\lvert\xi\rvert)^{-\sigma/2}.
\]
Since \(\sigma>2\), the restrictions of \(\wh\mu\) to all lines
through the origin are uniformly integrable, so every
one-dimensional projection of \(\mu\) has a uniformly bounded
density.  Since \(\sigma>k\), the same measure also has finite
\(I_j(\mu)\) for \(1\leq j\leq k\).  For a \(j\)-tuple of anchors, we
expand the averaged \(L^2\) norm of the squared-distance star measure.
Each anchor contributes, through the corresponding one-dimensional
projection, a factor \(O(\lvert z-z'\rvert^{-1})\); the product of the
\(j\) factors is therefore controlled by the energy kernel
\(\lvert z-z'\rvert^{-j}\).  Thus the star measure has an \(L^2\)
density for almost every anchor tuple.  Fubini and a fixed-pin
disintegration then build the complete pinned \(k\)-simplex
configuration measure.  The Salem hypothesis is used precisely to
place the uniform projection bound and the finite \(k\)-energy
condition on one and the same measure.

The gain from the Salem hypothesis occurs for intermediate-dimensional
simplices.  Theorem~\ref{thm:euclidean-codim-one}, followed by
coordinate marginalization, gives pinned absolute continuity at every
rank \(1\leq j\leq d-1\), but only under the common condition
\(\dim_{\mathrm H}(E)>d-1\).  Running its affine-span construction
directly at rank \(k\) does not lower this threshold: the cylindrical
\(L^2\) estimate requires only \(I_k(\mu)<\infty\), but absolute
continuity of the affine-span measure still requires the
inverse-volume condition \(s>d-1\).  The Salem argument averages
directly over the anchor tuples and avoids this bottleneck, yielding
the ambient-dimension-independent threshold
\(\dim_{\mathrm H}(E)>k\) for \(2\leq k\leq d-2\).  Both theorems give
absolute continuity of the complete squared-distance and
ordinary-distance laws for a full-\(\mu\)-measure set of pins; the
\(L^2\) conclusion is used for the star measures and is not asserted
for the final complete configuration density.  When \(k=d-1\), there
is no Salem-specific improvement, since
Theorem~\ref{thm:euclidean-codim-one} gives the same conclusion for
arbitrary compact sets.

\part{The finite field setting}
\label{part:finite-field}

Throughout Part~\ref{part:finite-field}, \(q\) is an odd prime power,  \(V=\Fq^d\), and \(Q\) is a fixed nondegenerate quadratic form on \(V\).

\section{Finite field preliminaries}
\label{sec:finite-preliminaries}

\subsection{Quadratic spaces and simplex classes}

We use the standard facts that every nondegenerate quadratic space of dimension at least three over a finite field of odd order is isotropic and that Witt's extension theorem holds in odd characteristic (see \cite[Section 7]{BHIPR} for
example).  When only the order of the group is relevant, \(\Orth_j(\Fq)\) denotes the orthogonal group of either nondegenerate \(j\)-dimensional quadratic form. Uniformly in the isometry class,
\begin{equation}\label{eq:orthogonal-size}
|\Orth_j(\Fq)|\approx_j q^{j(j-1)/2}.
\end{equation}

The following elementary lemma shows that degenerate simplex classes are negligible in positive-proportion estimates.

\begin{lemma} \label{lem:degenerate-classes}
For fixed \(2\leq k\leq d\), the number of ordered degenerate \(k\)-simplex congruence classes in a nondegenerate quadratic space $(V,Q)$ is $O_k(q^{t_k-1})$.
\end{lemma}

\begin{proof}
Let \(\mathbf x=(x_0,x_1,\ldots,x_k)\) be an ordered \(k\)-simplex, and let \(L:\Fq^k\to V\) be the linear map determined by its edge vectors:
\[
L(e_i)=x_i-x_0,\qquad 1\leq i\leq k,
\]
where $\{e_1,\cdots,e_k\}$ is the standard basis of $\Fq^k$. We separate two kinds of degeneracy.

First, suppose that \(\rank L=r< k\). The kernel is a \((k-r)\)-dimensional subspace of \(\Fq^k\), and there are \(O_k(q^{r(k-r)})\) possible kernels. Once the kernel is fixed, the form induced on the \(r\)-dimensional quotient has at most \(q^{r(r+1)/2}\) possibilities. Two edge maps with the same kernel and the same induced form define an isometry between their images, which extends to \(V\) by Witt's theorem. The number of rank-\(r\) classes is therefore at most
\[
 O_k\!\left(q^{r(k-r)+r(r+1)/2}\right)
 =
 O_k\!\left(q^{rk-r(r-1)/2}\right) = O_k\!\left(q^{t_k-1}\right),
\]
since $r\leq k-1$.

Second, if \(\rank L=k\) but the edge Gram matrix is singular, its entries lie on the determinant hypersurface in the \(t_k\)-dimensional space of symmetric matrices. This hypersurface has \(O_k(q^{t_k-1})\) points. Witt's theorem again shows that each Gram matrix gives at most one congruence class.

Combining the two cases proves the lemma.
\end{proof}

For \(2\leq k\leq d-1\), every nondegenerate \(k\)-dimensional quadratic space embeds isometrically into \(V\): by the classification of quadratic forms over finite fields, one may choose a nondegenerate complement of dimension \(d-k\) with the required discriminant. Witt's extension theorem therefore identifies the nondegenerate ordered classes with the nonsingular symmetric \(k\times k\) Gram matrices. Consequently,
\[
\bigl|\cT^{d,\mathrm{nd}}_{k,Q}(V)\bigr|
=
q^{t_k}+O_k(q^{t_k-1})
\approx_k q^{t_k}.
\]
Thus, degenerate classes do not affect positive-proportion estimates.

\subsection{Point--hyperplane variance}

The sum in the next lemma is over distinct affine hyperplanes, with scalar multiples of a defining equation representing the same hyperplane.

\begin{lemma} 
\label{lem:hyperplane-variance}
For \(E\subset\Fq^d\),
\begin{equation}\label{eq:hyperplane-variance}
\sum_H\left(|E\cap H|-\frac{|E|}{q}\right)^2  = q^{d-1}|E|-\frac{|E|^2}{q},
\end{equation}
where the sum is over all affine hyperplanes in $\Fq^d$. 
\end{lemma}

\begin{proof}
Put $M_j=\frac{q^j-1}{q-1}$ for $(j=d-1,d)$. There are \(qM_d\) affine hyperplanes, every point belongs to \(M_d\)
of them, and every pair of distinct points belongs to \(M_{d-1}\) of
them. Hence
\[
        \sum_H|E\cap H|=|E|M_d
\]
and
\[
        \sum_H|E\cap H|^2
        =
        |E|M_d+|E|(|E|-1)M_{d-1}.
\]
Expanding the left-hand side of \eqref{eq:hyperplane-variance}, and
using
\[
        M_d-M_{d-1}=q^{d-1},
        \qquad
        M_{d-1}-\frac{M_d}{q}=-\frac1q,
\]
proves the identity.
\end{proof}

\subsection{A second-moment estimate for base simplices}

Let \(\chi\) be the canonical nontrivial additive character of \(\Fq^d\). Choose coordinates on \(V\), and use the unnormalized Fourier transform
\[
\wh f(\xi) = \sum_{x\in V}f(x)\chi(-x\cdot\xi),\qquad (\xi\in \Fq^d),
\]
where $x\cdot \xi$ denotes the standard dot product of $x$ and $\xi$. 
Write \(Q(x)=x^TAx\) and $Q^\vee(\xi)=\xi^TA^{-1}\xi$. For \(\theta\in\Orth(Q)\) and \(z\in V\), set
\[
 \nu_\theta(z)
 =
 |\{(u,v)\in E^2:u-\theta v=z\}|.
\]

\begin{lemma} \label{lem:BHIPR-motion-moment}
For every integer $r\geq 2$,
\begin{equation}\label{eq:BHIPR-motion-moment}
\sum_{\theta\in\Orth(Q)}\sum_{z\in V}\nu_\theta(z)^r
\ll_{d,r}
q^{\frac{d^2-(2r-1)d}{2}}|E|^{2r}
+
q^{\frac{d^2-d+2}{2}}|E|^r.
\end{equation}
\end{lemma}

\begin{proof}
First, we need the following estimate from \cite[(2.5)--(2.6)]{BHIPR}:
\begin{equation} \label{eq:BHIPR-motion-variance}
\sum_{\theta\in\Orth(Q)}\sum_{z\in V}
\left|\nu_\theta(z)-\frac{|E|^2}{q^d}\right|^2
\ll_d q^{\frac{d^2-d+2}{2}} |E|^2.
\end{equation}
Indeed, after removing the frequency zero, applying Plancherel's theorem, and decomposing $V\setminus\{0\}$ into $\theta^T$-orbits, the left-hand side of \eqref{eq:BHIPR-motion-variance} equals 
\begin{equation} \label{eq_deduce_further}
q^{-d}\sum\limits_{\xi\neq 0}|\wh{\1_E}(\xi)|^2\sum\limits_{\theta\in \Orth(Q)}|\wh{\1_E}(-\theta^T\xi)|^2 = q^{-d}\sum\limits_{\Omega}m_\Omega \Bigg(\sum\limits_{\xi\in \Omega}|\wh{\1_E}(\xi)|^2\Bigg)^2,
\end{equation}
where the outer sum is over all the orbits $\Omega$, and $m_\Omega:=|\text{Stab}_{\Orth(Q)}(\xi_\Omega)|$ for any chosen $\xi_\Omega\in \Omega$. If $Q^\vee(\xi_\Omega)\neq 0$, then $m_\Omega\approx_d |\Orth_{d-1}(\Fq)|$. If $Q^\vee(\xi_\Omega)= 0$ and $\xi_\Omega\neq 0$, then the stabilizer is a parabolic group of order $\approx_d q^{d-2}|\Orth_{d-2}(\Fq)|$. In both cases, one has $m_\Omega\ll_d q^{(d-1)(d-2)/2}$ by \eqref{eq:orthogonal-size}. Now the right-hand side of \eqref{eq_deduce_further} is 
\[
\ll_d q^{-d}\cdot q^{(d-1)(d-2)/2} \sum\limits_{\xi\neq 0}|\wh{\1_E}(\xi)|^2 \leq q^{\frac{d^2-d+2}{2}} |E|^2.
\]

Second, we derive the higher-moment estimate. Since \(0\leq\nu_\theta(z)\leq|E|\leq q^d\), for every fixed integer \(r\geq2\) and \(a=|E|^2/q^d\) one has
\[
x^r\leq 2^{r-1}(a^r+ |x-a|^r)\leq 2^{r-1}(a^r+|E|^{r-2}|x-a|^2)
\qquad (0\leq x \leq |E|).
\]
Summing this inequality with \(x=\nu_\theta(z)\) and using \eqref{eq:BHIPR-motion-variance} yields
\[
\sum_{\theta\in\Orth(Q)}\sum_{z\in V}\nu_\theta(z)^r
\ll_{d,r}
\frac{q^{d(d-1)/2}|E|^{2r}}{q^{d(r-1)}}
+
q^{\frac{d^2-d+2}{2}} |E|^r.
\]
The conclusion then follows.
\end{proof}

Let $E\subset V$ be given. For each nondegenerate ordered $(d-2)$-simplex class \(D\) represented in \(E^{d-1}\), denote
\[
\cY_D:=
 \{\mathbf y\in E^{d-1}:
   \mathbf y\text{ is nondegenerate and belongs to \(D\)}\},
 \qquad
 m_D:=|\cY_D|.
\]
Moreover, recall \(t_{d-2}=\binom{d-1}{2}\).

\begin{lemma} \label{lem:base-second-moment}
With the notation above,  \begin{equation}\label{eq:base-second-moment} \sum_Dm_D^2 \ll_d q^{-t_{d-2}}|E|^{2d-2} + q^{d(d-1)/2}|E|^{d-1}.
\end{equation}
In particular, if \(|E|\ge q^{d-1}\), then
\[
\sum_Dm_D^2 \ll_d  \frac{|E|^{2d-2}}{q^{t_{d-2}}}.
\]
\end{lemma}

\begin{proof}
Note that $\sum_D m_D^2$ counts the number of pairs $(\mathbf y,\mathbf y')\in E^{d-1}\times E^{d-1}$ of congruent nondegenerate ordered $(d-2)$-simplex bases. The spaces $W:=\Span(y_2-y_1,\ldots,y_{d-1}-y_1)$ and $W':=\Span(y_2'-y_1',\ldots,y_{d-1}'-y_1')$ are nondegenerate. By Witt's extension theorem, the isometry between $W'$ and $W$, sending $y_i'-y_1'$ to $y_i-y_1$ $(2\leq i\leq d-1)$, extends to some $\theta\in \Orth(Q)$. The set of such extensions is a coset of $H_W:=\{\phi\in \Orth(Q):\, \phi|_W=id_W\}\simeq \Orth(Q|_{W\perp})$, which has size \(\approx q\). When $\theta$ is determined, there is a unique $z\in V$ such that $z+\theta y_1'= y_1$, and then $z+\theta y_i'= y_i$ $(2\leq i\leq d-1)$. Combining Lemma \ref{lem:BHIPR-motion-moment}, 
\[
q\sum_Dm_D^2 \ll_d \sum_{\theta\in\Orth(Q)}\sum_{z\in V}\nu_\theta(z)^{d-1} \ll_d q^{\frac{-d^2+3d}{2}}|E|^{2d-2}
+
q^{\frac{d^2-d+2}{2}}|E|^{d-1}.
\]
Dividing both sides by $q$ leads to the first estimate. Finally, since $q^{d(d-1)/2}|E|^{d-1} \le q^{-t_{d-2}}|E|^{2d-2}$ when \(|E|\ge q^{d-1}\), the second estimate holds.
\end{proof}

\subsection{Isotropic differences}

Denote
\[
\mathcal S_{Q,0}:=\{v\in V:Q(v)=0\},\qquad \nu_{E,Q}(0):=|\{(x,y)\in E^2:Q(x-y)=0\}|. 
\]

The following estimate is standard.  The explicit Fourier transform of the zero quadric is given in \cite[Lemma~4.1]{KohShenExtension}.  A short deduction is included because we will use both the zero-frequency term and the number of
projective isotropic directions.

\begin{lemma}
\label{lem:isotropic-differences}
For every \(E\subset V\),
\[
\nu_{E,Q}(0)   \le  \frac{|E|^2}{q}+q^{\lfloor d/2\rfloor}|E|.
\]
If \(d\ge3\), then \(V\) contains at least \(q^{d-2}\) isotropic
directions.
\end{lemma}

\begin{proof}
Orthogonality and completion of squares give
\begin{equation} \label{eq_zero_radius_fourier}
\wh{\1_{\mathcal S_{Q,0}}}(\xi) 
= q^{-1}\sum\limits_{r\in \Fq}\sum\limits_{x\in V}\chi(rQ(x)-x\cdot \xi)
= q^{d-1}\1_{\xi=0} + q^{-1}\eta(\det A)G^d
 \sum_{r\ne0}\eta(r)^d  \chi\!\left(-\frac{Q^\vee(\xi)}{4r}\right),
\end{equation}
where \(\eta\) is the quadratic character, and $G=\sum_{t\in\Fq}\chi(t^2)$ is the quadratic Gauss sum, satisfying \(|G|=q^{1/2}\). The sum over $r$ has size at most $q^{1/2}$ when $d$ is odd and at most $q$ when $d$ is even. Denote $g=\1_{\mathcal S_{Q,0}}-q^{-1}\1_V$. It follows that 
\[
|\wh{g}(\xi)| = \left|q^{-1}\eta(\det A)G^d
 \sum_{r\ne0}\eta(r)^d  \chi\!\left(-\frac{Q^\vee(\xi)}{4r}\right)\right| \leq q^{\lfloor d/2 \rfloor}
\]
for every $\xi$. Now Fourier inversion and Plancherel's theorem give
\[
 \left|\nu_{E,Q}(0) - \frac{|E|^2}{q}\right| = \left|
 \frac{1}{q^d}\sum_{\xi\in V}
 |\wh{\1_E}(\xi)|^2
 \wh{g}(\xi)\right|
 \le
 q^{\lfloor d/2\rfloor}|E|.
\]
The first assertion then follows. 

Moreover, evaluating \eqref{eq_zero_radius_fourier} at $\xi=0$, yields
\[
 |\mathcal S_{Q,0}|
 =
 \begin{cases}
 q^{d-1},&d\ \text{odd},\\
 q^{d-1}+\epsilon_Q(q-1)q^{(d-2)/2},&d\ \text{even},
 \end{cases}
\]
where \(\epsilon_Q\in\{-1,1\}\). 
Every nonzero isotropic vector spans a unique isotropic line, and every such line contains exactly $q-1$ nonzero vectors. The exact formula above yields
\[
\frac{|\mathcal S_{Q,0}|-1}{q-1} \ge q^{d-2}
\]
for \(d\ge3\). The second assertion now follows. 
\end{proof}

\section{Proof of Theorem~\ref{thm:finite-main} for \texorpdfstring{\(k=d-1\)}{k=d-1}}
\label{sec:codim-one-proof}

In this section, we prove Theorem~\ref{thm:finite-main} in the case
\(k=d-1\).

\subsection{Extensions of a nondegenerate base}

\begin{lemma}
\label{lem:base-number}
There exists \(C_d>0\) such that, if \(|E|\ge C_dq^{d-1}\), then
\[
\sum_Dm_D\gg_d |E|^{d-1},\qquad \sum_Dm_D^2 \ll_d \frac{|E|^{2d-2}}{q^{t_{d-2}}}.
\]
\end{lemma}

\begin{proof}
The second estimate is Lemma~\ref{lem:base-second-moment}. For the first, note that $\sum_D m_D$ counts the number of non-degenerate ordered $(d-2)$-simplices $(y_1,\ldots,y_{d-1})$. Let us choose the vertices successively. There are $|E|$ choices of $y_1$. Assume that $y_1,\ldots,y_{j-1}$ $(2\leq j\leq d-1)$ have been chosen. Denote $W:=\Span(y_2-y_1,\ldots,y_{j-1}-y_1)$. Then \(W\) is nondegenerate and \(\dim W=j-2\). Write
\[
y_j-y_1=w+z, \qquad w\in W, \quad z\in W^\perp.
\]
Express $w=\sum_i c_i (y_{i+1}-y_1)$ and write $c=(c_2, \ldots, c_{j-1})^T$. Denote by $G_j$ the edge Gram matrix of $(y_1,\ldots,y_j)$. Then 
\[
G_j=\begin{pmatrix}
G_{j-1} &G_{j-1}c\\
c^T G_{j-1} & c^TG_{j-1} c+Q(z)
\end{pmatrix}.
\]
By the Schur complement formula, calculation reveals that \(\det G_j=(\det G_{j-1})\,Q(z)\). So the enlarged simplex is degenerate precisely when $Q(z)= 0$. This equation has at most $O(q^{d-j+1})$ solutions in the $(d-j+2)$-dimensional nondegenerate space \(W^\perp\). Each such value of \(z\) has \(|W|=q^{j-2}\) possible values of \(w\), so at most \(O_d(q^{d-1})\) ambient points are excluded. Now, the induction gives
\[
 \sum_Dm_D
 \ge
 |E|\bigl(|E|-O_d(q^{d-1})\bigr)^{d-2}
 \gg_d |E|^{d-1}
\]
after increasing \(C_d\).
\end{proof}

For \(\mathbf y\in\cY_D\), let
\[
 \cA_{\mathbf y}
 =
 \{x\in E:
   (x,y_1,\ldots,y_{d-1})
   \text{ is a nondegenerate \((d-1)\)-simplex}\}.
\]
Thus \(\cA_{\mathbf y}\) is the set of apices in \(E\) that extend \(\mathbf y\) to a nondegenerate \((d-1)\)-simplex.

\begin{lemma}
\label{lem:admissible-apex}
There exists \(C_d>0\) such that, if \(|E|\ge C_dq^{d-1}\), then, for every nondegenerate base \(\mathbf y\),
\[
        |\cA_{\mathbf y}|
        \ge
        |E|-O_d(q^{d-1})
        \gg_d |E|.
\]
\end{lemma}

\begin{proof}
Applying the argument in Lemma~\ref{lem:base-number} to $\mathbf y=(y_1,\ldots,y_{d-1})$ with the new
vertex \(x\), the conclusion then follows.
\end{proof}

\subsection{Energy of the distance vectors}

For a fixed nondegenerate base \(\mathbf y\), we now study the distribution of the distance vectors determined by the apices \(x\in\cA_{\mathbf y}\). The main quantity is the second moment of the multiplicities of these vectors. For $\mathbf y=(y_1,\ldots,y_{d-1})$, define
\[
 \rho_{\mathbf y}(x)
 =
 \bigl(Q(x-y_1),\ldots,Q(x-y_{d-1})\bigr)
 \in\Fq^{d-1}.
\]
For \(\mathbf a\in\Fq^{d-1}\), put
\[
 n_{\mathbf y}(\mathbf a)
 =
 |\{x\in\cA_{\mathbf y}:
       \rho_{\mathbf y}(x)=\mathbf a\}|,
 \qquad
 e_{\mathbf y}
 =
 \sum_{\mathbf a\in\Fq^{d-1}}
 n_{\mathbf y}(\mathbf a)^2.
\]

\begin{lemma}
\label{lem:isotropic-pairs}
Suppose that \(x\ne x'\), \(Q(x-x')=0\), and
\[
        Q(x-y_i)=Q(x'-y_i)
        \qquad
        (1\le i\le d-1).
\]
Then \((x,y_1,\ldots,y_{d-1})\) is degenerate.
\end{lemma}

\begin{proof}
Let $u=x'-x$. Since \(Q(u)=0\), equality of the two distances gives
\[
        B(u,y_i-x)=0
        \qquad
        (1\le i\le d-1).
\]
Hence all \(y_i-x\) lie in \(u^\perp\). Since \(u\ne0\) is isotropic,
\[
        \rad(u^\perp)=\langle u\rangle,
        \qquad
        \rank(B|_{u^\perp})=d-2.
\]
The Gram matrix of \(d-1\) vectors in \(u^\perp\) therefore has rank at most \(d-2\), and is singular.
\end{proof}

\begin{proposition} 
\label{prop:distance-energy}
There exists \(C_d>0\) such that, if \(|E|\ge C_dq^{d-1}\), then
\[
        \sum_D\sum_{\mathbf y\in\cY_D}e_{\mathbf y}
        \ll_d
        \frac{|E|^{d+1}}{q^{d-1}}.
\]
\end{proposition}

\begin{proof}
Expanding the left-hand side counts tuples $(\mathbf y,x,x')$ with $x,x'\in\cA_{\mathbf y}$ and $\rho_{\mathbf y}(x)=\rho_{\mathbf y}(x')$. The diagonal \(x=x'\) contributes at most $|E|^d$. By Lemma~\ref{lem:isotropic-pairs}, no off-diagonal pair with \(Q(x-x')=0\) contributes.

Suppose that \(Q(x-x')\ne0\). Equality of the distance vectors implies that every \(y_i\) lies on the perpendicular bisector
\[
 H(x,x')
 =
 \{y\in V: \, Q(x-y)=Q(x'-y)\}
 = \Big\{y\in V:\, B(y,x-x')=\frac{Q(x)-Q(x')}{2}\Big\}.
\]
Its direction space is \((x-x')^\perp\), which is nondegenerate. 

For every nondegenerate affine hyperplane \(H\), let
\[
r(H) = |\{(x,x')\in E^2: \, x\ne x', \ H(x,x')=H\}|.
\]
If \(H=\{y:B(y,u)=c\}\), where \(Q(u)\ne0\), then reflection in \(H\) is given by
\[
\sigma_H(y) = y+\frac{2(c-B(y,u))}{Q(u)}u.
\]
Thus \(H\) and \(x\) determine \(x'=\sigma_H(x)\) uniquely. Hence
\begin{equation}\label{eq:r-basic}
\max_Hr(H)\le |E|,\qquad \sum_Hr(H)\le |E|^2.
\end{equation}

It remains to bound $\sum_Hr(H)|E\cap H|^{d-1}$, where the sum is over all nondegenerate $H$. Write $\delta(H)=|E\cap H|-q^{-1}|E|$. Since $0\le |E\cap H|,\, q^{-1}|E|\le q^{d-1}$, Taylor's theorem gives
\[
 |E\cap H|^{d-1}
 \le
 (q^{-1}|E|)^{d-1}
 +
 (d-1)(q^{-1}|E|)^{d-2}\delta(H)
 +
C_d q^{(d-1)(d-3)}\delta(H)^2.
\]
for some constant $C_d>0$. 
By Lemma~\ref{lem:hyperplane-variance} and
\eqref{eq:r-basic},
\[
 \sum_H r(H)\delta(H)^2
 \le
 \bigl(\max_Hr(H)\bigr)
 \sum_H\delta(H)^2
 \le
 q^{d-1}|E|^2.
\]
Cauchy--Schwarz inequality also gives
\[
 \left|\sum_Hr(H)\delta(H)\right|
 \le
 \left(\sum_Hr(H)\right)^{1/2}
 \left(\sum_Hr(H)\delta(H)^2\right)^{1/2}
 \le
 q^{(d-1)/2}|E|^2.
\]
Consequently,
\begin{align*}
 \sum_Hr(H)|E\cap H|^{d-1}
 &\ll_d
 \frac{|E|^{d+1}}{q^{d-1}}
 +
 |E|^{d}q^{(3-d)/2}
 +
 q^{(d-1)(d-2)}|E|^2.
\end{align*}
The first term equals $q^{-d+1}|E|^{d+1}$. The hypothesis \(|E|\ge C_dq^{d-1}\), for some sufficiently large $C_d>0$, gives
\[
 |E|^dq^{(3-d)/2}
 \le
 \frac{|E|^{d+1}}{q^{d-1}},
 \qquad
 q^{(d-1)(d-2)}|E|^2
 \le
 \frac{|E|^{d+1}}{q^{d-1}},
\]
and it also absorbs the diagonal contribution \(|E|^d\). This proves
the proposition.
\end{proof}

\subsection{Completion of the proof}

\begin{proof}[Proof of Theorem~\ref{thm:finite-main} when
\(k=d-1\)]
For each base class \(D\), put $S_D  = \sum_{\mathbf y\in\cY_D}e_{\mathbf y}$, and choose one representative \(\mathbf y_D\in\cY_D\) such that $e_{\mathbf y_D} \le S_D/m_D$. Let \(L_D\) be the number of distinct distance vectors $\rho_{\mathbf y_D}(x)$ with $x\in\cA_{\mathbf y_D}$. Cauchy--Schwarz inequality and Lemma~\ref{lem:admissible-apex} give
\[
 L_D
 \ge 
 \frac{|\cA_{\mathbf y_D}|^2}{e_{\mathbf y_D}}
 \gg_d
 \frac{|E|^2m_D}{S_D}.
\]
For a fixed base class \(D\), the edge Gram matrix of the base together with \(\rho_{\mathbf y_D}(x)\) determines the full edge Gram matrix by polarization. Distinct distance vectors therefore give distinct full Gram matrices, while the base distance data, which are recoverable from the full Gram matrix, distinguish different classes \(D\).
It follows that
\begin{equation}\label{eq:selection}
 |\cT^{d,\mathrm{nd}}_{d-1,Q}(E)|\geq \sum_D L_D
 \gg_d
 |E|^2\sum_D\frac{m_D}{S_D}
 \ge
 |E|^2
 \frac{\bigl(\sum_D\sqrt{m_D}\bigr)^2}
      {\sum_DS_D}.
\end{equation}

H\"older's inequality and Lemma~\ref{lem:base-number} lead to
\[
 \left(\sum_D\sqrt{m_D}\right)^2
 \ge
 \frac{\left(\sum_Dm_D\right)^3}
      {\sum_Dm_D^2}\gg_d
 q^{t_{d-2}}|E|^{d-1}.
\]
Combining this estimate, Proposition~\ref{prop:distance-energy}, and \eqref{eq:selection}, we obtain
\[
 |\cT^{d,\mathrm{nd}}_{d-1,Q}(E)|
 \gg_d
 |E|^2\frac{q^{t_{d-2}}|E|^{d-1}}
      {|E|^{d+1}/q^{d-1}}
 =
 q^{t_{d-2}+d-1}
 =
 q^{\binom d2}.
\]
This proves Theorem~\ref{thm:finite-main} in the codimension-one case. The corresponding sharpness statement is included in the general odd-codimension construction in Section~\ref{sec:sharpness}.
\end{proof}

\section{Isotropic quotients and Proof of Theorem~\ref{thm:finite-main}}
\label{sec:quotient}

Throughout this section we assume \(2\le k\le d-2\). This restriction is needed for a nontrivial quotient step, since the quotient has dimension \(d-2\). 

\subsection{A large image in an isotropic quotient}

Let \(\ell\subset V\) be an isotropic line. Then $\ell\subset\ell^\perp$ and $\dim\ell^\perp=d-1$. For an affine coset \(H=a+\ell^\perp\), define
\[
 \pi_{\ell,H}:H\longrightarrow \ell^\perp/\ell,
 \qquad
 \pi_{\ell,H}(x)=x-a+\ell.
\]

\begin{lemma}\label{lem:large-projection}
Let \(d\ge3\) and \(E\subset V\). Suppose that $1\le M\le q^{d-2}/8$ and $|E|\ge8qM$. Then there are an isotropic line \(\ell\) and a coset \(H=a+\ell^\perp\) such that
\[
        |\pi_{\ell,H}(E\cap H)|>M.
\]
\end{lemma}

\begin{proof}
Assume the contrary. For any fixed  isotropic direction \(\ell\). Each of the $q$ cosets of $\ell^\perp$ is partitioned
into affine lines parallel to $\ell$, and the image size counts exactly the fibers meeting $E$. There are \(q\) cosets of \(\ell^\perp\), and inside each coset the fibers of \(\pi_{\ell,H}\) are exactly the affine lines parallel to \(\ell\). By assumption, at most \(qM\) such affine lines $\gamma$ meet \(E\).

Write \(n_\gamma:=|E\cap\gamma|\). Then $\sum_\gamma n_\gamma = |E|$. Cauchy--Schwarz inequality and the condition \(|E|\ge8qM\) give
\[
|\{(x,y)\in E^2:\, x\neq y, \, x-y\in \ell\}| = \sum_\gamma n_\gamma(n_\gamma-1)
 \ge \frac{|E|^2}{qM}-|E|
 \ge \frac{|E|^2}{2qM}.
\]
There are at least \(q^{d-2}\) isotropic directions by Lemma~\ref{lem:isotropic-differences}. A nonzero isotropic difference
spans a unique isotropic direction, so summing the last estimate over all isotropic directions yields
\[
\nu_{E,Q}(0)-|E| = \sum_\ell |\{(x,y)\in E^2:\, x\neq y, \, x-y\in \ell\}| \ge \frac{q^{d-3}|E|^2}{2M}.
\]
Since \(M\le q^{d-2}/8\), this lower bound is at least \(4|E|^2/q\). Also, using \(|E|\ge8qM\), 
\[
 \frac{q^{d-3}|E|^2/(2M)}{q^{\lfloor d/2\rfloor}|E|}
 \ge4q^{d-2-\lfloor d/2\rfloor}\ge4.
\]
It follows that 
\[
\nu_{E,Q}(0)\geq \max\left\{\frac{4|E|^2}{q},\, 4q^{\lfloor d/2\rfloor}|E|\right\}>\frac{|E|^2}{q}+q^{\lfloor d/2\rfloor}|E|,
\]
which contradicts Lemma~\ref{lem:isotropic-differences}.
\end{proof}

\subsection{One reduction step}

For an isotropic line $\ell$, the restriction of \(Q\) to \(\ell^\perp\) has radical \(\ell\). Therefore $Q_\ell(v+\ell)=Q(v)$ defines a nondegenerate quadratic form on the \((d-2)\)-dimensional space $W_\ell=\ell^\perp/\ell$. 

\begin{lemma}\label{lem:preserve-gram}
If \(x_0,\ldots,x_k\in H=a+\ell^\perp\), then
\[
 Q(x_i-x_j)
 =
 Q_\ell\!\left(
   \pi_{\ell,H}(x_i)-\pi_{\ell,H}(x_j)
 \right)
 \qquad(0\le i<j\le k).
\]
In particular, the edge Gram matrix is preserved by
\(\pi_{\ell,H}\).
\end{lemma}

\begin{proof}
If \(v\in\ell^\perp\) and \(e\) spans \(\ell\), then
\[
        Q(v+te)=Q(v)+2tB(v,e)+t^2Q(e)=Q(v).
\]
Thus \(Q\) is constant on every quotient fiber. So
\[
 Q(x_i-x_j)
 =
 Q_\ell(x_i-x_j+\ell)
 =
 Q_\ell\!\left(
   \pi_{\ell,H}(x_i)-\pi_{\ell,H}(x_j)
 \right)
 \qquad(0\le i<j\le k).
\]
The polarization shows that the Gram-matrix is also preserved by $\pi_{\ell,H}$.
\end{proof}

\begin{proposition}\label{prop:one-step}
Let \(d\ge4\), \(2\le k\le d-2\), and \(0\le\alpha<d-2\). Suppose that there are constants \(C_0,c_0>0\), independent of \(q\),
such that, for every nondegenerate $(d-2)$-dimensional quadratic space $(W,P)$ and every \(A\subset W\),
\[
 |A|\ge C_0q^\alpha
 \quad\Longrightarrow\quad
 |\cT^{d-2,\mathrm{nd}}_{k,P}(A)|
 \ge c_0q^{t_k}.
\]
Then there are constants \(C_1>0\), depending only on \(d,k,\alpha,C_0\), and \(c_0\), such that, for every nondegenerate
quadratic form \(Q\) on \(\Fq^d\) and every \(E\subset\Fq^d\),
\[
 |E|\ge C_1q^{\alpha+1}
 \quad\Longrightarrow\quad
 |\cT^{d,\mathrm{nd}}_{k,Q}(E)|
 \ge c_0q^{t_k}.
\]
\end{proposition}

\begin{proof}
Replacing \(C_0\) by \(\max\{C_0,1\}\), we may assume that \(C_0\ge1\). Put \(M=C_0q^\alpha\). Since \(\alpha<d-2\), for all
sufficiently large \(q\) we have \(M\le q^{d-2}/8\). If \(|E|\ge8qM\), Lemma~\ref{lem:large-projection} gives
\(\ell\) and \(H\) such that
\[
        A:=\pi_{\ell,H}(E\cap H),
        \qquad
        |A|>M.
\]
Applying the assumed condition to \(A\subset W_\ell\), we obtain $|\cT^{d-2,\mathrm{nd}}_{k,Q_\ell}(A)|  \ge c_0q^{t_k}$. For each quotient simplex, choose one representative of each vertex in \(E\cap H\). Lemma~\ref{lem:preserve-gram} shows that the lifted simplex is nondegenerate and has the same Gram matrix. Distinct quotient
classes have distinct Gram matrices, so they remain distinct after lifting. As a result, $|\cT^{d,\mathrm{nd}}_{k,Q}(E)| \ge c_0 q^{t_k}$ for sufficiently large $q$. The remaining bounded values of \(q\) are handled by increasing $C_0$ to some larger constant $C_1$.
\end{proof}

\subsection{Completion of the proof of Theorem ~\ref{thm:finite-main}}

The case \(k=d-1\) was proved in Section~\ref{sec:codim-one-proof}, so it remains to consider \(2\le k\le d-2\).

\begin{proof}[Proof of Theorem~\ref{thm:finite-main} when $2\leq k\leq d-2$]

Suppose first that \(d-k\) is even. Put $r=(d-k)/2$. In dimension \(k\), \cite[Theorem 1.5]{BHIPR} gives the base-case threshold $a_0=k-(k-1)/(k+1)$:
\[
|E|\geq C_k q^{a_0} \quad\Longrightarrow\quad
 |\cT^{k}_{k,Q}(E)|
 \geq c_k q^{t_k}
\]
for some $C_k,c_k>0$. By Lemma~\ref{lem:degenerate-classes}, a positive proportion of the nondegenerate classes are nondegenerate, i.e, $|\cT^{k,\mathrm{nd}}_{k,Q}(E)|
 \geq c_k q^{t_k}$. Apply Proposition~\ref{prop:one-step} successively \(r\) times. After \(j\)
steps the exponent is \(a_j=a_0+j\). At the \(j\)-th step, the ambient dimension before reduction is \(k+2j\), and the range condition in Proposition~\ref{prop:one-step} holds because
\[
 (k+2j-2)-a_{j-1} = j-1+\frac{k-1}{k+1}>0.
\]
The final exponent is
\[
 a_r =  k-\frac{k-1}{k+1}+\frac{d-k}{2}  = \frac{d+k}{2}-\frac{k-1}{k+1}.
\]

Now suppose that \(d-k\) is odd. Put $r=(d-k-1)/2$. The codimension-one case proved in Section~\ref{sec:codim-one-proof}, applied in ambient dimension \(k+1\), gives the initial exponent $a_0=k$. Apply Proposition~\ref{prop:one-step} successively \(r\) times. At the \(j\)-th step, the ambient dimension before reduction is \(k+1+2j\), while \(a_{j-1}=k+j-1\). Hence
\[
 (k+1+2j-2)-a_{j-1}=j>0,
\]
so the range condition is satisfied. The final exponent is
\[
a_r  = k+\frac{d-k-1}{2} = \frac{d+k-1}{2}.
\]
These are precisely the two values in \eqref{eq:beta}. In the even case, direct subtraction gives
\[
\left(d-\frac{d-1}{k+1}\right)-\beta_{d,k} = \frac{k-1}{k+1}\,\frac{d-k}{2}.
\]
In the odd case, it gives
\[
\left(d-\frac{d-1}{k+1}\right)-\beta_{d,k} = \frac{(k-1)(d-k-1)/2+1}{k+1}.
\]
This also verifies the gain formula stated after Theorem~\ref{thm:finite-main}.
\end{proof}

\section{Sharpness examples}\label{sec:sharpness}

We first identify the largest dimension of a subspace on which the
quadratic form has rank at most \(k-1\).

\begin{proposition}\label{prop:sharpness}
Let \(d\ge3\), let \(Q\) be a nondegenerate quadratic form on
\(V=\Fq^d\), and let \(2\le k\le d-1\). Then there is a subspace
\(W\subset V\) such that
\[
 \dim W=\left\lfloor\frac{d+k-1}{2}\right\rfloor,
 \qquad
 \rank(B|_W)=k-1.
\]
Every \(k\)-simplex contained in an affine translate of \(W\) is
degenerate. Moreover, no subspace of larger dimension can have
restricted rank at most \(k-1\).
\end{proposition}

\begin{proof}
Write $h:=\left\lfloor (d-k+1)/2\right\rfloor$. By Witt decomposition, one has $V\simeq \mathbb H^{\perp w}\perp V_{\mathrm{an}}$, where \(\mathbb H\) is a hyperbolic plane and \(V_{\mathrm{an}}\) is anisotropic. Over a finite field of odd order, \(\dim V_{\mathrm{an}}\le2\), and hence $w\ge\left\lfloor (d-1)/2\right\rfloor$. Since \(k\ge2\), we have $h\le\left\lfloor (d-1)/2\right\rfloor\le w$. We may therefore choose a nondegenerate split subspace \(S\simeq\mathbb H^{\perp h}\) and a maximal totally isotropic subspace \(R\subset S\), with \(\dim R=h\).

If \(d-k\) is odd, then \(\dim S=d-k+1\). Set \(U=S^\perp\), so \(\dim U=k-1\). If \(d-k\) is even, then \(\dim S=d-k\) and \(\dim S^\perp=k\). In this case choose an anisotropic vector \(v\in S^\perp\) and set $U=S^\perp\cap v^\perp$. Such a vector exists because \(S^\perp\) is nondegenerate. Since
\(Q(v)\ne0\), we have $S^\perp=\langle v\rangle\perp U$, so \(U\) is nondegenerate and \(\dim U=k-1\).

In both cases, put \(W=U\oplus R\). The summands are orthogonal,
\(U\) is nondegenerate, and \(R\) is totally isotropic. Thus
\[
        \rad(W)=R,
        \qquad
        \rank(B|_W)=\dim U=k-1.
\]
Furthermore,
\[
 \dim W=k-1+h
 =\left\lfloor\frac{d+k-1}{2}\right\rfloor.
\]

Let \(x_0,\ldots,x_k\) belong to an affine translate of \(W\). All
edge vectors \(x_i-x_0\), \(1\le i\le k\), lie in \(W\). Their Gram
matrix has rank at most \(\rank(B|_W)=k-1\), and hence its determinant
vanishes.

For the final assertion, let \(L\subset V\), put \(s=\dim L\), and
write \(\rho=\dim\rad(L)\). If \(\rank(B|_L)\le k-1\), then $s-\rho\le k-1$. Since \(\rad(L)\subset L^\perp\) and \(\dim L^\perp=d-s\), we also have \(\rho\le d-s\). Consequently,
\[
        s\le k-1+\rho\le k-1+d-s,
\]
and hence $s\le\left\lfloor (d+k-1)/2\right\rfloor$. The subspace \(W\) constructed above attains equality.
\end{proof}

\begin{proof}[Proof of the sharpness of Theorem \ref{thm:finite-main} when $d-k$ is odd.]
Let \(E=W\) be constructed in Proposition \ref{prop:sharpness}. Then
\[
        |E|=q^{\lfloor(d+k-1)/2\rfloor},
        \qquad
        \cT^{d,\mathrm{nd}}_{k,Q}(E)=\varnothing.
\]
For every \(\alpha<\lfloor(d+k-1)/2\rfloor\) and every fixed \(C>0\), this example satisfies \(|E|\ge Cq^\alpha\) for all sufficiently large \(q\). Thus no exponent below \(\lfloor(d+k-1)/2\rfloor\) can guarantee a positive proportion of nondegenerate \(k\)-simplex classes. If \(d-k\) is odd, this obstruction equals $(d+k-1)/2=\beta_{d,k}$, which proves the sharpness assertion in Theorem~\ref{thm:finite-main}. If \(d-k\) is even, the exponent in Theorem~\ref{thm:finite-main} lies $2/(k+1)$ above the geometric obstruction.
\end{proof}

\clearpage
\part{The Euclidean setting}
\label{part:euclidean}

\section{Euclidean preliminaries and the base--apex decomposition}
\label{sec:euclidean-preliminaries}

The proof of Theorem~\ref{thm:euclidean-codim-one} is guided by the same
base--apex architecture as the codimension-one finite field argument in
Section~\ref{sec:codim-one-proof}, but the discrete counting mechanism is
replaced by a geometric analytic one.  In the finite field setting, a
second-moment estimate controls the distribution of the base classes,
while point hyperplane variance controls collisions among the
apex-distance vectors.  In the Euclidean setting, these roles are played
by absolute continuity of the affine-span measure and an averaged
\(L^2\) estimate for a family of cylindrical projection maps.  Together
they yield an \(L^2\) squared distance star measure for almost every
full-rank anchor tuple.  Coordinate marginalization then propagates this
control to all lower ranks, and a fixed-pin induction assembles the full
pinned simplex configuration while retaining all previously constructed
edge data.

\subsection{Measures, Fourier transforms, and energies}

We write \(\cL^m\) for Lebesgue measure on \(\R^m\) and
\(\cH^s\) for \(s\)-dimensional Hausdorff measure. For a finite Borel
measure \(\mu\) on \(\R^d\), our Fourier transform convention is
\[
        \wh\mu(\xi)
        =
        \int_{\R^d}e^{-2\pi i x\cdot\xi}\,d\mu(x).
\]
For \(0<\alpha<d\), its \(\alpha\)-energy is
\begin{equation}
 I_\alpha(\mu)
 =
 \iint_{\R^d\times\R^d}
 |x-y|^{-\alpha}\,d\mu(x)\,d\mu(y).
\label{eq:riesz-energy}
\end{equation}
Whenever either side is finite, the Fourier representation of the
Riesz energy is
\begin{equation}
 I_\alpha(\mu)
 =
 c_{d,\alpha}
 \int_{\R^d}
 |\wh\mu(\xi)|^2|\xi|^{\alpha-d}\,d\xi.
\label{eq:riesz-energy-fourier}
\end{equation}
For \(\omega\in\Sph^{d-1}\), we denote the scalar projection in the
direction \(\omega\) by
\[
        \pi_\omega(x)=x\cdot\omega.
\]

A probability measure \(\mu\) is called \(s\)-Frostman if
\[
        \mu(B(x,r))\le C_\mu r^s
        \qquad (x\in\R^d,\ r>0).
\]
Such a measure satisfies \(I_\alpha(\mu)<\infty\) for every
\(0<\alpha<s\).

For the proof of Theorem~\ref{thm:euclidean-codim-one}, fix
\[
        d-1<s<\dim_{\mathrm H}(E)
\]
and choose an \(s\)-Frostman probability measure \(\mu\) supported on
\(E\):
\begin{equation}
        \mu(B(x,r))\le C_\mu r^s.
\label{eq:frostman}
\end{equation}
Since \(E\) is compact, \(\supp\mu\subset B(0,R_0)\) for some
\(R_0<\infty\). Choose a nonnegative radial function
\(\phi\in C_c^\infty(B(0,1))\) with integral one, and put
\begin{equation}
 \phi_\eps(x)=\eps^{-d}\phi(x/\eps),
 \qquad
 f_\eps=\mu*\phi_\eps,
 \qquad
 d\mu_\eps(x)=f_\eps(x)\,dx.
\label{eq:mollification}
\end{equation}
The measures \(\mu_\eps\), \(0<\eps<1\), are supported in a fixed
bounded ball.

\subsection{Grassmannians and simplex volume}

For \(1\le h<d\), let \(G(d,h)\) denote the Grassmannian of
\(h\)-dimensional linear subspaces of \(\R^d\), equipped with
the invariant probability measure \(d\gamma(V)\). We write \(P_V\) for
orthogonal projection onto \(V\), and \(d\cH_L^h\) for
\(h\)-dimensional Hausdorff measure on an affine \(h\)-plane \(L\).
The affine Grassmannian is
parametrized by
\[
 A(d,h)
 =
 \{L=a+V:V\in G(d,h),\ a\in V^\perp\},
\]
and is equipped with the invariant measure
\[
        dL=d\gamma(V)\,da.
\]
For an ordered tuple \(Y=(y_0,\ldots,y_h)\), define
\[
 \mathcal V_h(Y)
 =
 \det\!\bigl(
 ((y_i-y_0)\cdot(y_j-y_0))_{1\le i,j\le h}
 \bigr)^{1/2}.
\]
Thus \(\mathcal V_h(Y)\) is \(h!\) times the Euclidean
\(h\)-dimensional volume of the simplex with vertices
\(y_0,\ldots,y_h\). When \(\mathcal V_h(Y)>0\), we write
\[
        \operatorname{Aff}(Y)
        =
        \operatorname{aff}(y_0,\ldots,y_h)
        \in A(d,h).
\]

\subsection{Distance and squared-distance measures}

For an integer \(M\geq1\), write
\[
\operatorname{Sq}(t_1,\ldots,t_M)
=
(t_1^2,\ldots,t_M^2).
\]

\begin{lemma}
\label{lem:squaring-ac}
Let \(\Omega\) be a finite Borel measure supported in a bounded subset
of \([0,\infty)^M\), and let
\[
        \Theta=\operatorname{Sq}_\#\Omega.
\]
If \(\Theta\ll\cL^M\), then \(\Omega\ll\cL^M\).

Moreover, for \(1\leq j<d\), the distance tuples and the
squared-distance tuples of degenerate \(j\)-simplices form
Lebesgue-null subsets of \(\R^{t_j}\), where
\[
        t_j=\binom{j+1}{2}.
\]
\end{lemma}

\begin{proof}
Let \(K\subset[0,\infty)^M\) be a bounded cube containing
\(\supp\Omega\), and let \(N\subset\R^M\) be a Borel set with
\(\cL^M(N)=0\). Set
\[
A:=N\cap K.
\]
Since \(\operatorname{Sq}|_K\) is a homeomorphism from \(K\) onto
\(\operatorname{Sq}(K)\), the set \(\operatorname{Sq}(A)\) is Borel.
Moreover, \(\operatorname{Sq}|_K\) is Lipschitz, so
\[
\cL^M\bigl(\operatorname{Sq}(A)\bigr)=0.
\]
Because \(\operatorname{Sq}\) is injective on \([0,\infty)^M\) and
\(\supp\Omega\subset K\),
\[
\operatorname{Sq}^{-1}\bigl(\operatorname{Sq}(A)\bigr)\cap K=A.
\]
Hence
\[
\Omega(N)
=
\Omega(A)
=
\Theta\bigl(\operatorname{Sq}(A)\bigr)
=
0,
\]
and therefore \(\Omega\ll\cL^M\).

For the final assertion, fix one vertex of a \(j\)-simplex and write
\(q_{ab}\) for the squared edge variables. The Gram matrix of the
remaining \(j\) edge vectors has entries
\[
G_{ii}=q_{0i},
\qquad
G_{ih}=\frac{q_{0i}+q_{0h}-q_{ih}}{2}
\quad (i\neq h).
\]
Thus degeneracy is equivalent to
\[
P_j(q):=\det G(q)=0.
\]
The polynomial \(P_j\) is nonzero, since the realization
\(z_0=0\), \(z_i=e_i\) gives \(G=I_j\). Hence the degenerate
squared-distance tuples lie in a Lebesgue-null algebraic hypersurface.

For ordinary edge variables \(r=(r_{ab})_{a<b}\), define
\[
\widetilde P_j(r)
:=
P_j\bigl((r_{ab}^2)_{a<b}\bigr).
\]
This polynomial is also nonzero, and its zero set contains all
degenerate ordinary-distance tuples. Therefore both degenerate loci
have Lebesgue measure zero.
\end{proof}

\subsection{The codimension-one base--apex decomposition}

Throughout Sections~\ref{sec:hulls}--\ref{sec:assembly}, we put
\begin{equation}
        \ell=d-2,
        \qquad
        n=\ell+1=d-1,
        \qquad
        m=\binom d2.
\label{eq:ell-n}
\end{equation}
Thus a codimension-one simplex is written as a base
\[
        Y=(y_0,\ldots,y_\ell)
\]
with \(n\) vertices, together with an apex \(x\).

The cylindrical estimate captures the full vector of \(n\) squared
distances from the apex \(x\) to the vertices of the base \(Y\).
The remaining task is to propagate this star-type control through the
lower-dimensional faces and then reassemble all edge coordinates of the
simplex without losing absolute continuity.

The next three sections carry out this program in three stages.  We
first prove absolute continuity for the distribution of the affine span
of the base.  We then establish an averaged \(L^2\) estimate for the
apex-distance vector over typical affine spans.  Finally, coordinate
marginalization transfers this control to every lower rank, and a
fixed-pin induction assembles these pieces into the complete pinned
squared-distance configuration measure.

\section{The distribution of affine spans}
\label{sec:hulls}

We begin the proof of Theorem~\ref{thm:euclidean-codim-one} by
controlling the affine span of a base distributed according to \(\mu^n\):
\[
Y=(y_0,\ldots,y_\ell).
\]
The affine Blaschke--Petkantschin formula reduces absolute continuity of
the induced affine-span measure to an inverse-square simplex-volume
estimate. The condition \(s>d-1\) ensures that this energy is finite,
uniformly under the mollification fixed in
Section~\ref{sec:euclidean-preliminaries}.

\subsection{Blaschke--Petkantschin formula and inverse volume estimates}

We use the notation of Section~\ref{sec:euclidean-preliminaries} with
\(h=\ell\).

The affine Blaschke--Petkantschin formula
\cite[Chapter~7]{SchneiderWeil} states that, for every nonnegative
measurable \(F\),
\begin{align}
\int_{(\R^d)^n}F(Y)\,dY
&=
c_{d,\ell}
\int_{A(d,\ell)}
\int_{L^n}
F(Y)\mathcal V_\ell(Y)^{d-\ell}\,
d\cH_L^\ell(y_0)\cdots d\cH_L^\ell(y_\ell)\,dL.
\label{eq:BP}
\end{align}
In our codimension-one simplex setting,
\[
d-\ell=2,
\]
and this exponent appears with opposite signs in the two formulas below.

For \(f\geq0\), define the affine \(\ell\)-plane transform
\[
R_\ell f(L)=\int_L f\,d\cH_L^\ell.
\]
Here and below, we interpret \(V_\ell(Y)^{-2}=+\infty\) when
\(V_\ell(Y)=0\). For \(M\geq 1\), apply \((7.1)\) to
\[
F_M(Y)
=
\min\{M,V_\ell(Y)^{-2}\}
\prod_{i=0}^{\ell} f(y_i).
\]
On each affine \(\ell\)-plane \(L\), the degenerate tuples form an
\((\cH^\ell_L)^n\)-null set. Hence, letting \(M\to\infty\) and using
monotone convergence on both sides, we obtain

\begin{align}
\|R_\ell f\|_{L^n(A(d,\ell))}^n
&=
c'_{d,\ell}
\int_{(\R^d)^n}
\mathcal V_\ell(Y)^{-2}
\prod_{i=0}^\ell f(y_i)\,dY.
\label{eq:inverse-BP}
\end{align}
Thus the \(L^n\) norm of the plane transform is governed by an
\emph{inverse} simplex-volume energy.

On the other hand, if
\[
\operatorname{Aff}(Y)
=
\operatorname{aff}(y_0,\ldots,y_\ell),
\]
then \eqref{eq:BP} gives the density of the affine-span pushforward:
\begin{align}
h_f(L)
&=
c_{d,\ell}
\int_{L^n}
\mathcal V_\ell(Y)^2
\prod_{i=0}^\ell f(y_i)\,
d\cH_L^\ell(y_0)\cdots d\cH_L^\ell(y_\ell).
\label{eq:hull-density-main}
\end{align}
The contrast between \eqref{eq:inverse-BP} and
\eqref{eq:hull-density-main} is important.  The factor
\(\mathcal V_\ell^{-2}\) controls concentration near degenerate bases,
while \(\mathcal V_\ell^2\) is the Jacobian appearing in the actual
affine-span distribution.

We now prove the inverse volume estimate needed later.

For an affine \(j\)-plane \(P\) and \(r>0\), write
\[
N_r(P)
=
\{x\in\R^d:\dist(x,P)<r\}.
\]

\begin{lemma}
\label{lem:plane-potential}
Let \(0\leq j\leq d-3\).  There is a constant \(C<\infty\) such that,
for every affine \(j\)-plane \(P\) meeting \(B(0,R_0+1)\),
\begin{equation}
\int \dist(x,P)^{-2}\,d\mu(x)
\leq C.
\label{eq:plane-potential}
\end{equation}
For the mollified measures \(\mu_\eps\) defined in
\eqref{eq:mollification}, the same estimate holds uniformly in \(P\)
and \(0<\eps<1\).
\end{lemma}

\begin{proof}
Since \(\supp\mu\) is contained in a fixed bounded set, the portion of
\(N_r(P)\) meeting \(\supp\mu\) can be covered by \(O(r^{-j})\) balls
of radius \(O(r)\).  The Frostman estimate \eqref{eq:frostman} therefore
gives
\begin{equation}
\mu(N_r(P))
\lesssim
r^{s-j}.
\label{eq:tube-frostman}
\end{equation}

A dyadic decomposition according to the distance from \(P\) yields
\[
\int \dist(x,P)^{-2}\,d\mu(x)
\lesssim
1+
\sum_{\nu\geq0}
2^{2\nu}\mu(N_{2^{-\nu}}(P)).
\]
Using \eqref{eq:tube-frostman}, the series is bounded by
\[
1+
\sum_{\nu\geq0}
2^{-\nu(s-j-2)}.
\]
Since \(j\leq d-3\),
\[
s-j
\geq
s-d+3
>
2,
\]
and the series converges uniformly in \(P\).

We next consider the mollified measures.  When \(r\geq\eps\),
\[
\mu_\eps(N_r(P))
\leq
\mu\bigl(N_{r+C\eps}(P)\bigr)
\lesssim
r^{s-j}.
\]
For \(r<\eps\), the Frostman condition gives
\[
\|f_\eps\|_\infty
\lesssim
\eps^{s-d},
\]
and hence
\[
\mu_\eps(N_r(P))
\lesssim
\eps^{s-d}r^{d-j}.
\]
Therefore the contribution from scales \(r<\eps\) to the layer-cake
integral is bounded by
\[
\eps^{s-d}
\int_0^\eps r^{d-j-3}\,dr
\lesssim
\eps^{s-j-2}
\lesssim
1,
\]
since \(s-j>2\).  The contribution from \(r\geq\eps\) is uniformly
bounded by the same estimate as above.  This proves the claim.
\end{proof}

The preceding lemma can now be iterated along the Gram--Schmidt
factorization of simplex volume.

\begin{proposition}
\label{prop:inverse-volume}
One has
\begin{equation}
J_\ell(\mu)
:=
\int_{(\R^d)^n}
\mathcal V_\ell(Y)^{-2}\,d\mu^n(Y)
<\infty,
\label{eq:inverse-volume}
\end{equation}
and
\begin{equation}
\sup_{0<\eps<1}
J_\ell(\mu_\eps)
<\infty.
\label{eq:inverse-volume-mollified}
\end{equation}
In particular, \(\mu^n\)-almost every base
\(Y=(y_0,\ldots,y_\ell)\) is affinely independent.
\end{proposition}

\begin{proof}
Gram--Schmidt gives the exact factorization
\begin{equation}
\mathcal V_\ell(y_0,\ldots,y_\ell)
=
\prod_{j=1}^\ell
\dist\!\left(
y_j,\operatorname{aff}(y_0,\ldots,y_{j-1})
\right).
\label{eq:GS-volume}
\end{equation}
Hence
\[
\mathcal V_\ell(Y)^{-2}
=
\prod_{j=1}^\ell
\dist\!\left(
y_j,\operatorname{aff}(y_0,\ldots,y_{j-1})
\right)^{-2}.
\]

For \(\nu\in\{\mu,\mu_\eps\}\) and \(M\geq1\), define
\[
J_{\ell,M}(\nu)
:=
\int
\prod_{j=1}^\ell
\min\!\left\{
M,\,
\dist\!\left(
y_j,\operatorname{aff}(y_0,\ldots,y_{j-1})
\right)^{-2}
\right\}
\,d\nu^n(Y).
\]

We integrate successively in \(y_\ell,y_{\ell-1},\ldots,y_1\).
After fixing \(y_0,\ldots,y_{j-1}\), let
\[
P_0=\operatorname{aff}(y_0,\ldots,y_{j-1}).
\]
If \(\dim P_0<j-1\), choose an affine \((j-1)\)-plane \(P\) containing
\(P_0\); otherwise set \(P=P_0\). Since \(P_0\subset P\),
\[
\dist(x,P_0)^{-2}\leq \dist(x,P)^{-2}.
\]
Moreover, \(P\) contains \(y_0\), so it meets the common support ball,
and
\[
j-1\leq \ell-1=d-3.
\]
Lemma~\ref{lem:plane-potential} therefore gives
\[
\int
\min\!\left\{
M,\,
\dist\!\left(
y_j,\operatorname{aff}(y_0,\ldots,y_{j-1})
\right)^{-2}
\right\}
\,d\nu(y_j)
\leq C,
\]
uniformly in the preceding tuple and in \(M\), and also uniformly in
\(\eps\) when \(\nu=\mu_\eps\).

Tonelli's theorem and successive integration now yield
\[
J_{\ell,M}(\nu)\leq C^\ell.
\]
Letting \(M\to\infty\) and applying monotone convergence gives
\[
J_\ell(\mu)<\infty
\]
and
\[
\sup_{0<\eps<1}J_\ell(\mu_\eps)<\infty.
\]
Thus \eqref{eq:inverse-volume} and
\eqref{eq:inverse-volume-mollified} follow.

Finally, \(\mathcal V_\ell(Y)=0\) precisely when the vertices of \(Y\)
are affinely dependent. Since
\(\mathcal V_\ell(Y)^{-2}=+\infty\) on this set, the finiteness of
\(J_\ell(\mu)\) implies
\[
\mu^n\{Y:\mathcal V_\ell(Y)=0\}=0.
\]
Hence \(\mu^n\)-almost every base is nondegenerate.
\end{proof}

\subsection{Absolute continuity of the affine span distribution}

We now study the distribution of affine spans generated by \(\mu^n\).
Fix a linear \(\ell\)-plane \(L_\ast\in A(d,\ell)\) through the origin,
and extend the affine-span map to all tuples by setting
\[
\operatorname{Aff}(Y):=L_\ast
\qquad\text{whenever }\mathcal V_\ell(Y)=0.
\]
On the open set of nondegenerate tuples, this agrees with the usual
affine span. The resulting map
\[
\operatorname{Aff}:(\R^d)^n\longrightarrow A(d,\ell)
\]
is Borel measurable and is continuous at every nondegenerate tuple.
By Proposition~\ref{prop:inverse-volume}, the degenerate locus has
\(\mu^n\)-measure zero, so the chosen value there does not affect the
pushforward measure \(\operatorname{Aff}_\#\mu^n\).

We will prove that \(\operatorname{Aff}_\#\mu^n\) is absolutely continuous
with respect to the invariant measure on \(A(d,\ell)\). The argument
proceeds by mollifying \(\mu\), obtaining uniform \(L^n\) bounds for the
associated \(\ell\)-plane transforms, and then passing to the limit.

\begin{proposition}
\label{prop:hull-ac}
The pushforward
\[
\eta=\operatorname{Aff}_\#\mu^n
\]
is absolutely continuous with respect to the invariant measure \(dL\)
on \(A(d,\ell)\).
\end{proposition}

\begin{proof}
Let \(f_\eps\) denote the density of \(\mu_\eps\), and set
\[
g_\eps=R_\ell f_\eps.
\]
By \eqref{eq:inverse-BP} and
\eqref{eq:inverse-volume-mollified},
\[
\sup_{0<\eps<1}
\|g_\eps\|_{L^n(A(d,\ell))}
<\infty.
\]
Thus the family \(\{g_\eps\}\) is weakly precompact in
\(L^n(A(d,\ell))\).

We first identify its weak limit.  Define the distribution \(R_\ell\mu\)
by
\begin{equation}
\langle R_\ell\mu,\Phi\rangle
=
\int_{\R^d}
\int_{G(d,\ell)}
\Phi\bigl(V,P_{V^\perp}x\bigr)\,
d\gamma(V)\,d\mu(x),
\qquad
\Phi\in C_c^\infty(A(d,\ell)).
\label{eq:distributional-plane-transform}
\end{equation}
Since \(\mu_\eps\to\mu\) weakly, the distributions \(R_\ell f_\eps\)
converge to \(R_\ell\mu\). Weak compactness therefore identifies the unique distributional limit
with a function
\[
g=R_\ell\mu\in L^n(A(d,\ell)).
\]

Define a finite Borel measure \(\Lambda\) on the affine-Grassmann bundle
by
\[
d\Lambda(V,a)
=
d\gamma(V)\,
d\bigl((P_{V^\perp})_\#\mu\bigr)(a).
\]
By the definition of \(R_\ell\mu\), equation
\eqref{eq:distributional-plane-transform} identifies \(\Lambda\) with
the measure
\[
g(V,a)\,d\gamma(V)\,da.
\]
Uniqueness of disintegration with respect to the projection
\[
(V,a)\longmapsto V
\]
therefore yields
\begin{equation}
g(V,a)\,da
=
(P_{V^\perp})_\#\mu
\qquad
\text{for \(\gamma\)-almost every \(V\)}.
\label{eq:projection-density}
\end{equation}  

For \(\gamma\)-almost every \(V\), equation
\eqref{eq:projection-density} identifies
\[
g(V,\cdot)\,da=(P_{V^\perp})_\#\mu.
\]
Since projection commutes with convolution, we have
\begin{equation}
g_\eps(V,\cdot)
=
\psi_{\eps,V}*_{V^\perp}g(V,\cdot),
\qquad
\psi_{\eps,V}(w)
=
\int_V \phi_\eps(v+w)\,d\cH_V^\ell(v).
\label{eq:fiber-convolution}
\end{equation}
Because \(\phi\) is radial, \(\psi_{\eps,V}\) is an
\(L^1(V^\perp)\)-normalized approximate identity on the
two-dimensional space \(V^\perp\), and its definition is independent
of the choice of orthonormal coordinates on \(V^\perp\).

Hence, for \(\gamma\)-almost every \(V\),
\[
\|g_\eps(V,\cdot)-g(V,\cdot)\|_{L^n(V^\perp)}
\longrightarrow 0.
\]
Moreover, convolution is contractive on \(L^n(V^\perp)\), so
\[
\|g_\eps(V,\cdot)-g(V,\cdot)\|_{L^n(V^\perp)}^n
\leq
2^n\|g(V,\cdot)\|_{L^n(V^\perp)}^n.
\]
Since
\[
\int_{G(d,\ell)}
\|g(V,\cdot)\|_{L^n(V^\perp)}^n\,d\gamma(V)
=
\|g\|_{L^n(A(d,\ell))}^n
<\infty,
\]
dominated convergence gives
\begin{equation}
g_\eps\longrightarrow g
\qquad\text{strongly in }L^n(A(d,\ell)).
\label{eq:strong-radon}
\end{equation}

We now return to the affine-span pushforwards.  Let
\[
\eta_\eps=\operatorname{Aff}_\#\mu_\eps^n.
\]
By \eqref{eq:hull-density-main}, \(\eta_\eps\) has density \(h_\eps\)
given by
\[
h_\eps(L)
=
c_{d,\ell}
\int_{L^n}
\mathcal V_\ell(Y)^2
\prod_{i=0}^\ell f_\eps(y_i)\,
d\cH_L^\ell(y_0)\cdots d\cH_L^\ell(y_\ell).
\]
Since all \(\mu_\eps\) are supported in one fixed bounded ball,
\[
\mathcal V_\ell(Y)^2
\leq C_{d,R_0}
\]
on the relevant support.  Hence
\begin{equation}
0\leq h_\eps(L)
\leq
C_{d,R_0}\,g_\eps(L)^n.
\label{eq:hull-domination}
\end{equation}
By \eqref{eq:strong-radon},
\[
g_\eps^n\longrightarrow g^n
\qquad\text{in }L^1(A(d,\ell)).
\]
Consequently, the family \(\{g_\eps^n\}\) is uniformly integrable, and
\eqref{eq:hull-domination} implies the same for the densities
\(\{h_\eps\}\).

It remains to pass to the limit in the affine-span measures.  Extend
\(\operatorname{Aff}\) arbitrarily on the degenerate locus.  By
Proposition~\ref{prop:inverse-volume}, \(\mu^n\)-almost every tuple is
nondegenerate, and hence \(\operatorname{Aff}\) is continuous at
\(\mu^n\)-almost every point. Since
\[
\mu_\eps^n\Longrightarrow \mu^n,
\]
the continuous mapping theorem for maps that are continuous almost
everywhere with respect to the limiting measure gives
\[
\eta_\eps\Longrightarrow\eta
\]
weakly.

Let
\[
K=
\{a+V\in A(d,\ell): |a|\leq R_0+1\}.
\]
The measure \(\eta\), and all \(\eta_\eps\) for sufficiently small
\(\eps\), are supported in \(K\).

Now let \(N\subset A(d,\ell)\) be Borel with \(dL(N)=0\).  Given
\(\delta>0\), choose an open set \(U\supset N\cap K\), contained in a
slightly larger finite-measure region, such that
\[
dL(U)<\delta.
\]
By Portmanteau,
\[
\eta(N)
\leq
\eta(U)
\leq
\liminf_{\eps\to0}\eta_\eps(U).
\]
Since
\[
\eta_\eps(U)=\int_U h_\eps(L)\,dL
\]
and the family \(\{h_\eps\}\) is uniformly integrable, the right-hand
side tends to zero uniformly as \(dL(U)\to0\).  Hence
\[
\eta(N)=0.
\]
Therefore
\[
\eta\ll dL,
\]
as claimed.
\end{proof}

\section{Averaged \texorpdfstring{\(L^2\)}{L2} estimates for apex distances}
\label{sec:cylindrical}

We next study the distances determined by one additional point and the
fixed \(n\)-tuple \(Y\).  Once the affine span \(L\) of \(Y\) is fixed,
these distances can be encoded by the following family of auxiliary
maps.  For
\(L=a+V\in A(d,\ell)\), let \(W=V^\perp\), so that \(\dim W=2\), and define
\begin{equation}
C_L(x)
=
\bigl(P_Vx,\lvert P_Wx-a\rvert^2\bigr)
\in V\times\R.
\label{eq:cylinder-map}
\end{equation}
We equip \(V\times\R\) with the product measure
\(d\cH_V^\ell(z)\,dr\).  If a pushforward measure below does not admit an
\(L^2\) density, its \(L^2\) norm is understood to be \(+\infty\).

The following estimate controls the measures \((C_L)_\#\mu\) in \(L^2\),
on average over \(L\).  After expanding the \(L^2\) norm, the resulting
kernel is bounded by a constant multiple of the Riesz kernel of order
\(\ell+1=d-1\), matching the dimension hypothesis in
Theorem~\ref{thm:euclidean-codim-one}.

\subsection{An averaged \texorpdfstring{\(L^2\)}{L2} estimate}

\begin{proposition}
\label{prop:cylinder-average}
For every \(R<\infty\),
\begin{equation}
\int_{G(d,\ell)}
\int_{\substack{a\in V^\perp\\ \lvert a\rvert\leq R}}
\bigl\lVert(C_{a+V})_\#\mu\bigr\rVert_{L^2(V\times\R)}^2
\,da\,d\gamma(V)
\lesssim_{d,R}
I_{\ell+1}(\mu).
\label{eq:cylinder-average}
\end{equation}
In particular, \((C_L)_\#\mu\) has an \(L^2\) density for
\(dL\)-almost every \(L\in A(d,\ell)\).
\end{proposition}

\begin{proof}
We argue by regularization.  Let \(\Phi_{\delta,V}\) be a nonnegative,
compactly supported, \(L^1\)-normalized product approximate identity on
\(V\times\R\).  We use radial kernels in the \(V\)-variable, so the
definition is independent of the choice of orthonormal coordinates on
\(V\) and depends measurably on \(V\).  Its autocorrelation kernels
satisfy
\begin{equation}
K_\delta^{(\ell)}(z)
\lesssim
\delta^{-\ell}\mathbf1_{\{\lvert z\rvert\lesssim\delta\}},
\qquad
K_\delta^{(1)}(t)
\lesssim
\delta^{-1}\mathbf1_{\{\lvert t\rvert\lesssim\delta\}}.
\label{eq:kernel-bounds}
\end{equation}

Fix \(x,x'\in\R^d\), and write
\[
u=x-x',
\qquad
\bar x=\frac{x+x'}2.
\]
Expanding the smoothed \(L^2\) norm produces the pair kernel
\begin{equation}
K_\delta^{(\ell)}(P_Vu)\,
K_\delta^{(1)}
\!\left(
2P_Wu\cdot(P_W\bar x-a)
\right).
\label{eq:cylinder-collision}
\end{equation}
We estimate its average in \(V\) and \(a\).

Suppose first that \(\lvert u\rvert>C\delta\).  On the support of the
first factor in \eqref{eq:cylinder-collision},
\[
\lvert P_Vu\rvert\lesssim\delta,
\qquad
\lvert P_Wu\rvert\approx\lvert u\rvert.
\]
The relevant Grassmannian estimate is
\begin{equation}
\gamma\left\{
V\in G(d,\ell):
\lvert P_Vu\rvert\lesssim\delta
\right\}
\lesssim
\left(\frac{\delta}{\lvert u\rvert}\right)^\ell.
\label{eq:crofton}
\end{equation}
Indeed, for \(V\) distributed according to \(d\gamma\), rotational
invariance shows that the random variable
\[
\frac{\lvert P_Vu\rvert^2}{\lvert u\rvert^2}
\]
has the \(\operatorname{Beta}(\ell/2,1)\) distribution, since
\(d-\ell=2\).  Therefore its distribution function at
\((C\delta/\lvert u\rvert)^2\) is bounded by a constant multiple of
\((\delta/\lvert u\rvert)^\ell\), proving \eqref{eq:crofton}.

We next integrate in the offset variable.  For every \(V\) satisfying
the support condition above,
\begin{equation}
\int_{\substack{a\in W\\ \lvert a\rvert\leq R}}
K_\delta^{(1)}
\!\left(
2P_Wu\cdot(P_W\bar x-a)
\right)\,da
\lesssim_R
\lvert u\rvert^{-1}.
\label{eq:offset-coarea}
\end{equation}
To see this directly, set
\[
w=P_Wu,
\qquad
e=\frac{w}{\lvert w\rvert},
\]
and write
\[
a=\tau e+b,
\qquad
b\in W\cap e^\perp.
\]
Since \(\|K_\delta^{(1)}\|_1=1\), integrating first in \(\tau\) gives
\[
\int_{\lvert a\rvert\leq R}
K_\delta^{(1)}
\!\left(
2w\cdot(P_W\bar x-a)
\right)\,da
\leq
\frac{2R}{2\lvert w\rvert}
\|K_\delta^{(1)}\|_1
\lesssim_R
\lvert w\rvert^{-1}.
\]
On the support under consideration,
\(\lvert w\rvert\approx\lvert u\rvert\), and hence
\eqref{eq:offset-coarea} follows.

Combining \eqref{eq:kernel-bounds},
\eqref{eq:crofton}, and \eqref{eq:offset-coarea}, the average of
\eqref{eq:cylinder-collision} is bounded by
\begin{equation}
C_R
\lvert u\rvert^{-\ell}
\lvert u\rvert^{-1}
=
C_R
\lvert u\rvert^{-(\ell+1)}.
\label{eq:far-collision}
\end{equation}

It remains to consider the near-diagonal regime
\(\lvert u\rvert\leq C\delta\).  There we simply use the trivial
bounds in \eqref{eq:kernel-bounds}.  After integration in \(V\) and
\(a\), this gives
\[
C_R\delta^{-(\ell+1)}
\lesssim
C_R\lvert u\rvert^{-(\ell+1)}.
\]
Thus the average of the kernel in \eqref{eq:cylinder-collision}, over \(V\)
and \(a\), is \(O_R(\lvert u\rvert^{-(\ell+1)})\).

Fubini's theorem, together with
\eqref{eq:far-collision} and the near-diagonal estimate above,
therefore yields, uniformly in \(\delta\),
\[
\int_{G(d,\ell)}
\int_{\lvert a\rvert\leq R}
\bigl\|
(C_{a+V})_\#\mu*\Phi_{\delta,V}
\bigr\|_2^2
\,da\,d\gamma(V)
\lesssim_R
\iint
\lvert x-x'\rvert^{-(\ell+1)}
\,d\mu(x)\,d\mu(x').
\]
The right-hand side is precisely \(I_{\ell+1}(\mu)\), which is finite
because
\[
\ell+1=d-1<s.
\]

Now choose a sequence \(\delta_\nu\downarrow0\). For
\(L=a+V\in A(d,\ell)\), set
\[
A_\nu(L)
:=
\bigl\|
(C_L)_\#\mu*\Phi_{\delta_\nu,V}
\bigr\|_{L^2(V\times\mathbb R)}^2.
\]
The approximate identity and weak lower
semicontinuity give
\[
\bigl\|(C_L)_\#\mu\bigr\|_{L^2(V\times\mathbb R)}^2
\leq
\liminf_{\nu\to\infty}A_\nu(L),
\]
where the left-hand side is interpreted as \(+\infty\) when
\((C_L)_\#\mu\) does not admit an \(L^2\) density.

Consequently, Fatou's lemma yields
\[
\begin{aligned}
&\int_{G(d,\ell)}
\int_{\substack{a\in V^\perp\\ |a|\leq R}}
\bigl\|
(C_{a+V})_\#\mu
\bigr\|_{L^2(V\times\mathbb R)}^2
\,da\,d\gamma(V)
\\
&\qquad\leq
\int_{G(d,\ell)}
\int_{\substack{a\in V^\perp\\ |a|\leq R}}
\liminf_{\nu\to\infty}A_\nu(a+V)
\,da\,d\gamma(V)
\\
&\qquad\leq
\liminf_{\nu\to\infty}
\int_{G(d,\ell)}
\int_{\substack{a\in V^\perp\\ |a|\leq R}}
A_\nu(a+V)
\,da\,d\gamma(V)
\\
&\qquad\lesssim_{d,R}
I_{\ell+1}(\mu).
\end{aligned}
\]
This proves \eqref{eq:cylinder-average}. Taking \(R\) through the
positive integers shows that \((C_L)_\#\mu\) has an \(L^2\) density
for \(dL\)-almost every \(L\in A(d,\ell)\).  
\end{proof}

\subsection{From the auxiliary map to apex distances}

We next show that, for a nondegenerate base \(Y\), the map \(C_L\)
determines the full vector of squared distances from an additional point
to the vertices of \(Y\).

Let
\[
Y=(y_0,\ldots,y_\ell)
\]
be a nondegenerate base, and write
\[
L=\operatorname{Aff}(Y)=a+V,
\qquad
y_i=a+b_i,
\qquad
b_i\in V.
\]
Define the squared apex-distance map
\begin{equation}
S_Y(x)
=
\bigl(
\lvert x-y_0\rvert^2,\ldots,
\lvert x-y_\ell\rvert^2
\bigr)
\in\R^n.
\label{eq:apex-map-main}
\end{equation}

The next lemma shows that \(C_L(x)\) determines \(S_Y(x)\).

\begin{lemma}
\label{lem:cylinder-star}
There exists a global \(C^\infty\) diffeomorphism
\[
T_Y:V\times\R\longrightarrow\R^n
\]
such that
\begin{equation}
S_Y=T_Y\circ C_L.
\label{eq:star-factor}
\end{equation}
Its Jacobian is constant and nonzero, with
\begin{equation}
\lvert\det DT_Y\rvert
=
2^\ell\mathcal V_\ell(Y).
\label{eq:star-jacobian}
\end{equation}
\end{lemma}

\begin{proof}
For \(z\in V\) and \(r\in\R\), define
\[
T_Y(z,r)
=
\bigl(
\lvert z-b_i\rvert^2+r
\bigr)_{i=0}^\ell.
\]
Since \(V\perp V^\perp\), the Pythagorean identity gives immediately
\[
S_Y=T_Y\circ C_L,
\]
which is \eqref{eq:star-factor}.

To prove that \(T_Y\) is globally invertible, subtract the zeroth
coordinate from the others:
\begin{equation}
T_Y(z,r)_i-T_Y(z,r)_0
=
-2z\cdot(b_i-b_0)
+\lvert b_i\rvert^2-\lvert b_0\rvert^2.
\label{eq:recover-z}
\end{equation}
Because the base \(Y\) is nondegenerate, the vectors
\[
b_1-b_0,\ldots,b_\ell-b_0
\]
form a basis of \(V\).  Hence the \(\ell\) differences in
\eqref{eq:recover-z} determine \(z\) uniquely by a linear system.  Once
\(z\) is known, the zeroth coordinate determines
\[
r
=
T_Y(z,r)_0-\lvert z-b_0\rvert^2.
\]
Thus \(T_Y\) is globally one-to-one and onto, with smooth inverse.

Finally, subtracting the zeroth row of \(DT_Y\) from each of the
remaining rows reduces the Jacobian determinant to the determinant of
the matrix with rows
\[
-2(b_i-b_0),
\qquad
1\leq i\leq \ell.
\]
Taking absolute values gives
\[
\lvert\det DT_Y\rvert
=
2^\ell\mathcal V_\ell(Y),
\]
which proves \eqref{eq:star-jacobian}.
\end{proof}

We can now transfer the averaged estimate to the apex-distance map.

\begin{corollary}
\label{cor:good-bases}
For $\mu^n$-almost every base $Y$,
\begin{equation}
\frac{d(S_Y)_\#\mu}{dt}\in L^2(\R^n),
\qquad\text{and hence}\qquad
(S_Y)_\#\mu\ll\cL^n.
\label{eq:apex-ac}
\end{equation}
\end{corollary}

\begin{proof}
Fix a sequence \(\delta_\nu\downarrow0\), and for
\(L=a+V\in A(d,\ell)\), define
\[
A_\nu(L)
:=
\bigl\|
(C_L)_\#\mu*\Phi_{\delta_\nu,V}
\bigr\|_{L^2(V\times\mathbb R)}^2.
\]
For each \(\nu\), the map
\[
L\longmapsto A_\nu(L)
\]
is Borel measurable. Indeed, expanding the square as in
\eqref{eq:cylinder-collision} expresses \(A_\nu(L)\) as the integral,
with respect to \(d\mu(x)\,d\mu(x')\), of a bounded Borel function of
\((L,x,x')\).

Define
\[
A_{\mathrm{good}}
:=
\left\{
L\in A(d,\ell):
\liminf_{\nu\to\infty}A_\nu(L)<\infty
\right\}.
\]
Then \(A_{\mathrm{good}}\) is Borel. By the approximate
\(L^2\) identity,
\[
L\in A_{\mathrm{good}}
\]
if and only if \((C_L)_\#\mu\) has an
\(L^2(V\times\mathbb R)\) density. Proposition~\ref{prop:cylinder-average}
therefore implies that \(A_{\mathrm{good}}\) has full \(dL\)-measure
in every bounded offset region. By Proposition~\ref{prop:cylinder-average},
\(A_{\mathrm{good}}\) has full \(dL\)-measure in every bounded offset
region.

Since \(\supp\mu\subset B(0,R)\), whenever
\(Y\in(\supp\mu)^n\) is nondegenerate and
\[
L=\Aff(Y)=a+V,
\]
we have \(y_0\in L\), and therefore
\[
|a|=\dist(0,L)\leq |y_0|\leq R.
\]

Moreover, Proposition~\ref{prop:hull-ac} gives
\[
\Aff_\#\mu^n\ll dL.
\]
It follows that
\[
\Aff(Y)\in A_{\mathrm{good}}
\]
for \(\mu^n\)-almost every \(Y\). Independently,
Proposition~\ref{prop:inverse-volume} shows that
\(\mu^n\)-almost every base is nondegenerate. Hence, for
\(\mu^n\)-almost every \(Y\), both
\[
(C_{\Aff(Y)})_\#\mu
\]
has an \(L^2\) density and
\[
\mathcal V_\ell(Y)>0.
\]

Fix such a base \(Y\), and write
\[
L=\Aff(Y)=a+V.
\]
By Lemma~\ref{lem:cylinder-star},
\[
S_Y=T_Y\circ C_L,
\]
where
\[
T_Y:V\times\R\longrightarrow\R^n
\]
is a smooth diffeomorphism with constant Jacobian
\[
|\det DT_Y|=2^\ell\mathcal V_\ell(Y)>0.
\]

Let \(f_L\in L^2(V\times\R)\) be the density of \((C_L)_\#\mu\).
Since
\[
(S_Y)_\#\mu
=
(T_Y)_\#\bigl((C_L)_\#\mu\bigr),
\]
the measure \((S_Y)_\#\mu\) has density
\[
g_Y(t)
=
\frac{f_L(T_Y^{-1}t)}
     {2^\ell\mathcal V_\ell(Y)}.
\]
Using the change of variables \(t=T_Y(z,r)\), we obtain
\[
\begin{aligned}
\|g_Y\|_{L^2(\R^n)}^2
&=
\frac{1}{\bigl(2^\ell\mathcal V_\ell(Y)\bigr)^2}
\int_{\R^n}
\bigl|f_L(T_Y^{-1}t)\bigr|^2\,dt \\
&=
\frac{1}{2^\ell\mathcal V_\ell(Y)}
\int_{V\times\R}|f_L(z,r)|^2\,dz\,dr \\
&=
\frac{1}{2^\ell\mathcal V_\ell(Y)}
\|f_L\|_{L^2(V\times\R)}^2
<\infty.
\end{aligned}
\]
Therefore \((S_Y)_\#\mu\) has an \(L^2(\R^n)\) density for
\(\mu^n\)-almost every base \(Y\). This proves
\eqref{eq:apex-ac}.
\end{proof}

\section{The pinned codimension-one configuration measure (Theorem \ref{thm:euclidean-codim-one})}
\label{sec:assembly}

We now pass from the full-rank star estimate in
Corollary~\ref{cor:good-bases} to complete pinned simplex
configurations.  The first step takes coordinate marginals of the
full-rank star measure.  The second uses Fubini to choose one common
set of pins and then adds the remaining vertices one at a time.

For \(1\leq j\leq n\) and
\[
Y=(y_0,\ldots,y_{j-1})\in(\R^d)^j,
\]
define
\[
S_Y^{(j)}(z)
=
\bigl(
\lvert z-y_0\rvert^2,\ldots,\lvert z-y_{j-1}\rvert^2
\bigr)
\in\R^j
\]
and
\[
\rho_Y^{(j)}=(S_Y^{(j)})_\#\mu.
\]
Thus \(S_Y^{(n)}=S_Y\) in the notation of
\eqref{eq:apex-map-main}.

\begin{lemma}[]
\label{lem:codim-marginal-descent}
For every $1\le j\le n$, there exists a Borel set
\[
G_j\subset (\operatorname{supp}\mu)^j,
\qquad
\mu^j(G_j)=1,
\]
such that $\rho_Y^{(j)}$ has an $L^2(\mathbb R^j)$ density for every
$Y\in G_j$.
\end{lemma}

\begin{proof}
Fix a sequence $\varepsilon_\nu\downarrow 0$. For each $1\le j\le n$,
let $\Psi_{\varepsilon,j}$ be a nonnegative, $C_c^\infty$,
$L^1$-normalized product approximate identity on $\mathbb R^j$.

For each fixed $\nu$, the map
\[
Y\longmapsto
\bigl\|
\rho_Y^{(j)}*\Psi_{\varepsilon_\nu,j}
\bigr\|_{L^2(\mathbb R^j)}^2
\]
is Borel measurable. Indeed, expanding the square gives
\[
\bigl\|
\rho_Y^{(j)}*\Psi_{\varepsilon_\nu,j}
\bigr\|_{L^2(\mathbb R^j)}^2
=
\iint
K_{\varepsilon_\nu,j}
\bigl(
S_Y^{(j)}(z)-S_Y^{(j)}(z')
\bigr)
\,d\mu(z)\,d\mu(z'),
\]
where
\[
K_{\varepsilon_\nu,j}
=
\Psi_{\varepsilon_\nu,j}
*
\widetilde{\Psi}_{\varepsilon_\nu,j},
\qquad
\widetilde{\Psi}_{\varepsilon_\nu,j}(t)
=
\Psi_{\varepsilon_\nu,j}(-t).
\]
For fixed $\nu$, the integrand is a bounded continuous function of
$(Y,z,z')$. Hence the displayed map is Borel measurable.

Define
\[
G_j
=
\left\{
Y\in(\operatorname{supp}\mu)^j:
\liminf_{\nu\to\infty}
\bigl\|
\rho_Y^{(j)}*\Psi_{\varepsilon_\nu,j}
\bigr\|_{L^2(\mathbb R^j)}
<\infty
\right\}.
\]
Thus $G_j$ is Borel.

We first observe that, for any finite Borel measure $\nu_0$ on
$\mathbb R^j$,
\[
\liminf_{\nu\to\infty}
\bigl\|
\nu_0*\Psi_{\varepsilon_\nu,j}
\bigr\|_{L^2(\mathbb R^j)}
<\infty
\]
if and only if $\nu_0$ has an $L^2(\mathbb R^j)$ density.

Indeed, suppose first that
\[
d\nu_0=f\,dt,
\qquad
f\in L^2(\mathbb R^j).
\]
By Young's inequality,
\[
\bigl\|
\nu_0*\Psi_{\varepsilon_\nu,j}
\bigr\|_2
=
\bigl\|
f*\Psi_{\varepsilon_\nu,j}
\bigr\|_2
\le
\|f\|_2,
\]
so the liminf is finite.

Conversely, suppose that
\[
\liminf_{\nu\to\infty}
\bigl\|
\nu_0*\Psi_{\varepsilon_\nu,j}
\bigr\|_2
<\infty.
\]
Then there exists a subsequence, still denoted by
$\varepsilon_\nu$, such that
\[
\sup_\nu
\bigl\|
\nu_0*\Psi_{\varepsilon_\nu,j}
\bigr\|_2
<\infty.
\]
By weak compactness in $L^2(\mathbb R^j)$, after passing to a further
subsequence there exists $f\in L^2(\mathbb R^j)$ such that
\[
\nu_0*\Psi_{\varepsilon_\nu,j}
\rightharpoonup f
\qquad
\text{weakly in }L^2(\mathbb R^j).
\]
On the other hand, since $\Psi_{\varepsilon_\nu,j}$ is an approximate
identity,
\[
\nu_0*\Psi_{\varepsilon_\nu,j}
\longrightarrow \nu_0
\]
in the sense of distributions. Hence
\[
d\nu_0=f\,dt.
\]
Therefore $G_j$ is precisely the set of $j$-tuples $Y$ for which
$\rho_Y^{(j)}$ has an $L^2(\mathbb R^j)$ density.

By Corollary~8.3,
\[
\mu^n(G_n)=1.
\]

Fix $1\le j<n$, and let
\[
p_j:(\operatorname{supp}\mu)^n
\longrightarrow
(\operatorname{supp}\mu)^j
\]
be the projection onto the first $j$ coordinates.

Suppose
\[
\widetilde Y=(Y,Y')\in G_n.
\]
Then $\rho_{\widetilde Y}^{(n)}$ has an
$L^2(\mathbb R^n)$ density, say $f_{\widetilde Y}$, and
$\rho_Y^{(j)}$ is the marginal of $\rho_{\widetilde Y}^{(n)}$
onto the first $j$ coordinates.

Set
\[
D=\operatorname{diam}(\operatorname{supp}\mu),
\qquad
Q=[0,1+D^2].
\]
Since every squared distance between two points of
$\operatorname{supp}\mu$ lies in $[0,D^2]$, we may take
$f_{\widetilde Y}$ to vanish outside $Q^n$.

The marginal $\rho_Y^{(j)}$ therefore has density
\[
g_Y(t)
=
\int_{Q^{n-j}}
f_{\widetilde Y}(t,u)\,du,
\qquad
t\in\mathbb R^j.
\]
By the Cauchy--Schwarz inequality,
\[
|g_Y(t)|^2
\le
|Q|^{n-j}
\int_{Q^{n-j}}
|f_{\widetilde Y}(t,u)|^2\,du.
\]
Integrating in $t$ gives
\[
\|g_Y\|_{L^2(\mathbb R^j)}^2
\le
|Q|^{n-j}
\|f_{\widetilde Y}\|_{L^2(\mathbb R^n)}^2
<\infty.
\]
Hence $Y\in G_j$. Therefore
\[
G_n\subset p_j^{-1}(G_j).
\]
It follows that
\[
1
=
\mu^n(G_n)
\le
\mu^n\bigl(p_j^{-1}(G_j)\bigr)
=
\mu^j(G_j).
\]
Since $\mu^j(G_j)\le1$, we conclude that
\[
\mu^j(G_j)=1.
\]
This proves the lemma.
\end{proof}

\begin{proof}[Proof of Theorem~\ref{thm:euclidean-codim-one}]
Let \(\mu\) be the \(s\)-Frostman probability measure fixed in
\eqref{eq:frostman}, where
\[
d-1<s<\dim_{\mathrm H}(E).
\]

For \(j=1\), interpret \(\mu^0\) as unit mass on a one-point space.
For every \(1\leq j\leq n\), define
\[
E_j
:=
\left\{
x\in\supp\mu:
\mu^{j-1}
\left(
\left\{
(z_1,\ldots,z_{j-1}):
(x,z_1,\ldots,z_{j-1})\notin G_j
\right\}
\right)
=0
\right\}.
\]
Since \(G_j\) is Borel, the corresponding section-measure function is
Borel, and hence \(E_j\) is Borel. By Fubini's theorem and
Lemma~\ref{lem:codim-marginal-descent},
\[
\mu(E_j)=1.
\]
Thus, for every \(x\in E_j\),
\begin{equation}
\mu^{j-1}
\left(
\left\{
(z_1,\ldots,z_{j-1}):
(x,z_1,\ldots,z_{j-1})\notin G_j
\right\}
\right)
=0.
\label{eq:codim-good-fixed-pin}
\end{equation}

Set
\[
E_\mu
:=
\bigcap_{j=1}^n E_j.
\]
Then \(E_\mu\subset\supp\mu\) is Borel and
\[
\mu(E_\mu)=1.
\]

Fix \(x\in E_\mu\), and write
\[
\lambda_{j,x}
:=
(\Phi_{j,x})_\#\mu^j.
\]
We prove by induction on \(j\) that
\[
\lambda_{j,x}\ll\cL^{t_j},
\qquad 1\leq j\leq n.
\]

For \(j=1\), condition~\eqref{eq:codim-good-fixed-pin} says that
\[
(x)\in G_1.
\]
Hence, by the defining property of \(G_1\),
\[
\lambda_{1,x}
=
\left(
z\longmapsto |z-x|^2
\right)_\#\mu
=
\rho^{(1)}_{(x)}
\ll\cL^1.
\]

Now suppose that \(2\leq j\leq n\) and
\[
\lambda_{j-1,x}\ll\cL^{t_{j-1}}.
\]
For
\[
Z=(z_1,\ldots,z_{j-1}),
\]
define
\[
\rho_{x,Z}
:=
\left(
z\longmapsto
\bigl(
|z-x|^2,
|z-z_1|^2,
\ldots,
|z-z_{j-1}|^2
\bigr)
\right)_\#\mu.
\]
By \eqref{eq:codim-good-fixed-pin},
\[
(x,Z)\in G_j
\]
for \(\mu^{j-1}\)-almost every \(Z\). Therefore, by the defining
property of \(G_j\),
\[
\rho_{x,Z}\ll\cL^j
\]
for \(\mu^{j-1}\)-almost every \(Z\).

After a fixed permutation of coordinates, identify
\[
\R^{t_j}
=
\R^{t_{j-1}}\times\R^j,
\qquad
t_j=t_{j-1}+j.
\]
Let
\[
N\subset\R^{t_{j-1}}\times\R^j
\]
be a Borel set of Lebesgue measure zero. For
\(b\in\R^{t_{j-1}}\), write
\[
N_b
:=
\{a\in\R^j:(b,a)\in N\}.
\]
By Fubini's theorem,
\[
\cL^j(N_b)=0
\]
for \(\cL^{t_{j-1}}\)-almost every \(b\). Since
\[
(\Phi_{j-1,x})_\#\mu^{j-1}
=
\lambda_{j-1,x}
\ll\cL^{t_{j-1}},
\]
it follows that
\[
\cL^j
\left(
N_{\Phi_{j-1,x}(Z)}
\right)
=0
\]
for \(\mu^{j-1}\)-almost every \(Z\).

For \(\mu^{j-1}\)-almost every such \(Z\), we also have
\(\rho_{x,Z}\ll\cL^j\). Hence
\[
\rho_{x,Z}
\left(
N_{\Phi_{j-1,x}(Z)}
\right)
=0.
\]
Integrating first in the new vertex gives
\[
\begin{aligned}
\lambda_{j,x}(N)
&=
\int
\rho_{x,Z}
\left(
N_{\Phi_{j-1,x}(Z)}
\right)
\,d\mu^{j-1}(Z)\\
&=0.
\end{aligned}
\]
Therefore
\[
\lambda_{j,x}\ll\cL^{t_j}.
\]
This completes the induction.

Taking \(j=n=d-1\) and recalling that \(t_n=m\), we obtain
\[
(\Phi_{d-1,x})_\#\mu^{d-1}
\ll\cL^m
\qquad
(x\in E_\mu).
\]

Let \(\Omega_x\) be the corresponding pinned ordinary-distance
configuration measure. Since \(E\) is compact, \(\Omega_x\) is
supported in a bounded subset of \([0,\infty)^m\), and
\[
(\Phi_{d-1,x})_\#\mu^{d-1}
=
\operatorname{Sq}_\#\Omega_x.
\]
Lemma~\ref{lem:squaring-ac} therefore gives
\[
\Omega_x\ll\cL^m.
\]

The same lemma shows that the ordinary-distance tuples of degenerate
\((d-1)\)-simplices form a Lebesgue-null set. Since
\(\Delta_{d-1,x}(E)\) is compact and the degenerate configurations are
characterized by the vanishing of the Gram-determinant polynomial
\[
\widetilde P_{d-1}(r)
:=
P_{d-1}
\bigl(
(r_{ab}^2)_{0\leq a<b\leq d-1}
\bigr),
\]
the set
\[
\Delta^{\mathrm{nd}}_{d-1,x}(E)
=
\Delta_{d-1,x}(E)
\cap
\left\{
r\in\R^m:
\widetilde P_{d-1}(r)\neq0
\right\}
\]
is Borel.

Moreover, the degenerate locus has \(\cL^m\)-measure zero, so
\(\Omega_x\) assigns zero mass to that locus. Since \(\Omega_x\) is a
probability measure,
\[
\Omega_x
\bigl(
\Delta^{\mathrm{nd}}_{d-1,x}(E)
\bigr)
=
1.
\]
Finally, since \(\Omega_x\ll\cL^m\),
\[
\cL^m
\bigl(
\Delta^{\mathrm{nd}}_{d-1,x}(E)
\bigr)
>0
\qquad
(x\in E_\mu).
\]
\end{proof}

Having established the codimension one result for arbitrary compact sets, we now turn to intermediate-dimensional simplices, where the affine-span method does not yield the desired threshold and the Fourier decay of Salem measures provides a different analytic input.

\section{Pinned simplices in Salem sets (Theorem~\ref{thm:pinned-salem})}
\label{sec:pinned-salem}

\subsection{An averaged \texorpdfstring{\(L^2\)}{L2} estimate for pinned distance maps}

This section proves Theorem~\ref{thm:pinned-salem}.  Recall that if
\(E\) is a compact Salem set and
\[
0<\sigma<\dim_{\mathrm H}(E),
\]
then there is a probability measure \(\mu\), supported on \(E\), such
that
\begin{equation}
\lvert\wh\mu(\xi)\rvert
\leq
C_\sigma(1+\lvert\xi\rvert)^{-\sigma/2}.
\label{eq:salem-decay-pinned}
\end{equation}

The proof has two steps.  We first obtain an averaged \(L^2\) estimate
for the distances from one point to a fixed \(j\)-tuple.  We then use
this estimate inductively, keeping one point fixed throughout, to build
the complete pinned simplex configuration measure.  This induction is
the complete-graph case of the graph-building argument in
\cite{BFOPR}; we include the full argument needed here.

Recall that, for $1\le j<d$,
\[
t_j=\binom{j+1}{2},
\]
and, for $x\in\R^d$ with $z_0=x$,
\[
\Phi_{j,x}(z_1,\ldots,z_j)
=
\bigl(|z_a-z_b|^2\bigr)_{0\le a<b\le j}.
\]

For
\[
Y=(y_0,\ldots,y_{j-1})\in(\R^d)^j,
\]
define
\[
S_Y(z)
=
\bigl(
\lvert z-y_0\rvert^2,\ldots,
\lvert z-y_{j-1}\rvert^2
\bigr)
\in\R^j,
\qquad
\rho_Y=(S_Y)_\#\mu.
\]

\begin{proposition}
\label{prop:all-rank-star}
Let \(\mu\) be a compactly supported probability measure on
\(\R^d\).  Suppose that every projection
\((\pi_\omega)_\#\mu\) has a density \(h_\omega\) and
\begin{equation}
P
:=
\sup_{\omega\in\Sph^{d-1}}
\|h_\omega\|_{L^\infty(\R)}
<\infty.
\label{eq:uniform-projection-pinned}
\end{equation}
If \(1\leq j<d\) and \(I_j(\mu)<\infty\), then
\(\rho_Y\) has an \(L^2(\R^j)\) density for
\(\mu^j\)-almost every \(Y\), and
\begin{equation}
\int
\left\|
\frac{d\rho_Y}{dt}
\right\|_2^2
\,d\mu^j(Y)
\leq
2^{-j}P^j I_j(\mu).
\label{eq:all-rank-star-bound}
\end{equation}

\end{proposition}

\begin{proof}
Choose a nonnegative even one-dimensional approximate identity
\(\varphi_\eps\), and set
\[
\Psi_{\eps,j}(t_1,\ldots,t_j)
=
\prod_{i=1}^j\varphi_\eps(t_i),
\qquad
K_\eps=\varphi_\eps*\varphi_\eps.
\]
Then \(K_\eps\geq0\) and \(\|K_\eps\|_1=1\).

Expanding the square and applying Fubini gives
\begin{align}
\int
\|\rho_Y*\Psi_{\eps,j}\|_2^2\,d\mu^j(Y)
&=
\iint
\prod_{i=0}^{j-1}
\left[
\int
K_\eps\!\left(
\lvert z-y_i\rvert^2
-
\lvert z'-y_i\rvert^2
\right)
d\mu(y_i)
\right]
d\mu(z)\,d\mu(z').
\label{eq:all-rank-star-expansion}
\end{align}

Fix \(z\neq z'\) and write
\[
u=z-z',
\qquad
\omega=\frac{u}{\lvert u\rvert},
\qquad
\bar z=\frac{z+z'}2.
\]
The identity
\begin{equation}
\lvert z-y\rvert^2-\lvert z'-y\rvert^2
=
2\lvert u\rvert
\bigl(\bar z\cdot\omega-y\cdot\omega\bigr)
\label{eq:bisector-linear}
\end{equation}
reduces the inner integral to a one-dimensional projection of \(\mu\).
Using \eqref{eq:uniform-projection-pinned},
\begin{align}
\int
K_\eps\!\left(
\lvert z-y\rvert^2-\lvert z'-y\rvert^2
\right)
d\mu(y)
&=
\int_\R
K_\eps\!\left(
2\lvert u\rvert(\bar z\cdot\omega-t)
\right)
h_\omega(t)\,dt \notag\\
&\leq
P\int_\R
K_\eps\!\left(
2\lvert u\rvert(\bar z\cdot\omega-t)
\right)\,dt \notag\\
&=
\frac{P}{2\lvert u\rvert}
\int_\R K_\eps(r)\,dr \notag\\
&=
\frac{P}{2\lvert u\rvert}.
\label{eq:one-anchor-bound}
\end{align}
Here we used the change of variables
\[
r=2\lvert u\rvert(\bar z\cdot\omega-t),
\qquad
\lvert dt\rvert=\frac{\lvert dr\rvert}{2\lvert u\rvert},
\]
together with \(\|h_\omega\|_\infty\leq P\) and
\(\|K_\eps\|_1=1\).

Since \(I_j(\mu)<\infty\), the diagonal has
\(\mu\times\mu\)-measure zero.  Substituting
\eqref{eq:one-anchor-bound} into
\eqref{eq:all-rank-star-expansion}, we obtain, uniformly in \(\eps\),
\[
\int
\|\rho_Y*\Psi_{\eps,j}\|_2^2\,d\mu^j(Y)
\leq
2^{-j}P^j
\iint
\lvert z-z'\rvert^{-j}\,d\mu(z)\,d\mu(z')
=
2^{-j}P^j I_j(\mu).
\]

Choose a sequence \(\eps_\nu\downarrow0\).  By Fatou's lemma,
for \(\mu^j\)-almost every \(Y\),
\[
\liminf_{\nu\to\infty}
\|\rho_Y*\Psi_{\eps_\nu,j}\|_2
<\infty.
\]
For such \(Y\), choose an \(L^2\)-bounded subsequence.  Weak compactness
then gives a further subsequence converging weakly in \(L^2(\R^j)\), while
approximate identity convergence identifies the distributional limit
with \(\rho_Y\).  Hence \(\rho_Y\) is represented by an \(L^2\)
function.  Weak lower semicontinuity, followed by Fatou's lemma in
\(Y\), proves \eqref{eq:all-rank-star-bound}.
\end{proof}

The proposition controls the distances from one new point to a fixed
\(j\)-tuple.  To recover all pairwise distances while keeping one point
fixed, we now add the remaining points one at a time and retain the
already constructed distance data.

\subsection{Induction with a fixed pin}

\begin{lemma}
\label{lem:pinned-star-induction}
Let \(\mu\) be a compactly supported probability measure on \(\R^d\),
and let \(1\leq k<d\).  Suppose that
\eqref{eq:uniform-projection-pinned} holds and
\[
I_j(\mu)<\infty
\qquad
(1\leq j\leq k).
\]
Then there exists a Borel set
\(E_\mu\subset\supp\mu\), with \(\mu(E_\mu)=1\), such that
\[
(\Phi_{k,x})_\#\mu^k
\ll
\cL^{t_k}
\qquad (x\in E_\mu).
\]
\end{lemma}

\begin{proof}
Choose a nonnegative even function
\(\varphi\in C_c^\infty(\mathbb R)\) with
$
\int_{\mathbb R}\varphi=1,$
and put
\[
\varphi_\varepsilon(t)
=
\varepsilon^{-1}\varphi(t/\varepsilon),
\qquad
\Psi_{\varepsilon,j}(t_1,\ldots,t_j)
=
\prod_{i=1}^j\varphi_\varepsilon(t_i).
\]

Choose a sequence \(\varepsilon_\nu\downarrow0\). For each
\(1\leq j\leq k\), define
\[
\mathcal G_j
=
\left\{
Y\in(\supp\mu)^j:
\liminf_{\nu\to\infty}
\|\rho_Y*\Psi_{\varepsilon_\nu,j}\|_2<\infty
\right\}.
\]

For fixed \(\nu\), the map
\[
Y\longmapsto
\|\rho_Y*\Psi_{\varepsilon_\nu,j}\|_2^2
\]
is Borel measurable. Indeed, if
\[
K_\varepsilon
=
\varphi_\varepsilon*\varphi_\varepsilon,
\]
then expanding the square gives
\[
\|\rho_Y*\Psi_{\varepsilon,j}\|_2^2
=
\iint
\prod_{i=0}^{j-1}
K_\varepsilon
\bigl(
|z-y_i|^2-|z'-y_i|^2
\bigr)
\,d\mu(z)\,d\mu(z').
\]
For fixed \(\varepsilon>0\), the integrand is a bounded continuous
function of \((Y,z,z')\). Hence the displayed map is Borel, and so
\(\mathcal G_j\) is Borel.

The approximate identity shows that
\[
Y\in\mathcal G_j
\quad\Longleftrightarrow\quad
\rho_Y
\text{ has an }L^2(\mathbb R^j)\text{ density}.
\]
Indeed, the forward implication follows by taking an
\(L^2\)-bounded subsequence, using weak compactness in \(L^2\), and
identifying the distributional limit with \(\rho_Y\). Conversely, if
\[
d\rho_Y=f_Y\,dt,
\qquad
f_Y\in L^2(\mathbb R^j),
\]
then Young's inequality gives
\[
\|\rho_Y*\Psi_{\varepsilon,j}\|_2
\leq
\|f_Y\|_2.
\]
Therefore Proposition~\ref{prop:all-rank-star} implies
\[
\mu^j(\mathcal G_j)=1.
\]

For \(x\in\supp\mu\) and \(1\leq j\leq k\), define
\[
q_j(x)
:=
\mu^{j-1}
\left(
\left\{
Z\in(\supp\mu)^{j-1}:
(x,Z)\notin\mathcal G_j
\right\}
\right),
\]
where, for \(j=1\), we interpret \(\mu^0\) as unit mass on a
one-point space. Since \(\mathcal G_j\) is Borel, the measurable-section
theorem shows that \(q_j\) is Borel.

Moreover, by Fubini's theorem and \(\mu^j(\mathcal G_j)=1\),
\[
\int_{\supp\mu} q_j(x)\,d\mu(x)
=
\mu^j\bigl((\supp\mu)^j\setminus\mathcal G_j\bigr)
=
0.
\]
Thus \(q_j(x)=0\) for \(\mu\)-almost every \(x\). Define
\[
E_\mu
:=
\supp\mu\cap
\bigcap_{j=1}^k\{x:q_j(x)=0\}.
\]
Then \(E_\mu\) is Borel and
\[
\mu(E_\mu)=1.
\]
In particular, for every \(x\in E_\mu\) and every \(1\leq j\leq k\),
\begin{equation}
\mu^{j-1}
\left(
\left\{
Z:(x,Z)\notin\mathcal G_j
\right\}
\right)
=0.
\label{eq:salem-good-fixed-pin}
\end{equation}

Fix \(x\in E_\mu\). We prove by induction on \(j\) that
\[
\lambda_{j,x}
:=
(\Phi_{j,x})_\#\mu^j
\ll
\cL^{t_j},
\qquad 1\leq j\leq k.
\]

For \(j=1\), \eqref{eq:salem-good-fixed-pin} says that
\((x)\in\mathcal G_1\). Hence
\[
\lambda_{1,x}
=
\left(
z\longmapsto |z-x|^2
\right)_\#\mu
\ll
\cL^1.
\]

Now let \(2\leq j\leq k\), and suppose that
\[
\lambda_{j-1,x}
=
(\Phi_{j-1,x})_\#\mu^{j-1}
\ll
\cL^{t_{j-1}}.
\]
For
\[
Z=(z_1,\ldots,z_{j-1}),
\]
define
\[
\rho_{x,Z}
:=
\left(
z\longmapsto
\bigl(
|z-x|^2,
|z-z_1|^2,
\ldots,
|z-z_{j-1}|^2
\bigr)
\right)_\#\mu.
\]
By \eqref{eq:salem-good-fixed-pin},
\[
\rho_{x,Z}\ll\cL^j
\]
for \(\mu^{j-1}\)-almost every \(Z\).

After a fixed permutation of coordinates, identify
\[
\R^{t_j}
=
\R^{t_{j-1}}\times\R^j,
\qquad
t_j=t_{j-1}+j.
\]
Let
\[
N\subset\R^{t_{j-1}}\times\R^j
\]
be a Borel set of Lebesgue measure zero. For
\(b\in\R^{t_{j-1}}\), write
\[
N_b
:=
\{a\in\R^j:(b,a)\in N\}.
\]
By Fubini's theorem,
\[
\cL^j(N_b)=0
\]
for \(\cL^{t_{j-1}}\)-almost every \(b\). Since
\[
\lambda_{j-1,x}\ll\cL^{t_{j-1}},
\]
we have
\[
\cL^j
\left(
N_{\Phi_{j-1,x}(Z)}
\right)
=0
\]
for \(\mu^{j-1}\)-almost every \(Z\).

For almost every such \(Z\), we also have
\(\rho_{x,Z}\ll\cL^j\). Therefore
\[
\rho_{x,Z}
\left(
N_{\Phi_{j-1,x}(Z)}
\right)
=0.
\]
Integrating first in the new vertex gives
\[
\begin{aligned}
\lambda_{j,x}(N)
&=
\int
\rho_{x,Z}
\left(
N_{\Phi_{j-1,x}(Z)}
\right)
\,d\mu^{j-1}(Z)\\
&=0.
\end{aligned}
\]
Hence
\[
\lambda_{j,x}\ll\cL^{t_j}.
\]

This completes the induction and proves the lemma upon taking \(j=k\).
\end{proof}

\begin{corollary}
\label{cor:pinned-measure-criterion}
Let \(\mu\) be a compactly supported probability measure on
\(\R^d\), and let \(1\leq k<d\). Suppose that
\[
I_k(\mu)<\infty
\]
and
\begin{equation}
\sup_{\omega\in S^{d-1}}
\int_{\R}
|\widehat{\mu}(r\omega)|\,dr
<\infty.
\label{eq:line-fourier-condition}
\end{equation}
Then there exists a Borel set
\(E_\mu\subset\supp\mu\), with \(\mu(E_\mu)=1\), such that, for every
\(x\in E_\mu\), both the pinned squared-distance configuration measure
\[
(\Phi_{k,x})_\#\mu^k
\]
and the corresponding pinned ordinary-distance configuration measure
are absolutely continuous with respect to \(\cL^{t_k}\). In
particular,
\[
\cL^{t_k}
\bigl(
\Delta^{\mathrm{nd}}_{k,x}(\supp\mu)
\bigr)>0.
\]
\end{corollary}

\begin{proof}
For \(\omega\in S^{d-1}\), the Fourier transform of the projected
measure satisfies
\[
\widehat{(\pi_\omega)_\#\mu}(r)
=
\widehat{\mu}(r\omega).
\]
Hence \eqref{eq:line-fourier-condition} and Fourier inversion imply
that \((\pi_\omega)_\#\mu\) has a bounded continuous density
\(h_\omega\), with
\[
\sup_{\omega\in S^{d-1}}
\|h_\omega\|_{L^\infty(\R)}
<\infty.
\]

Moreover, \(I_k(\mu)<\infty\) implies
\[
I_j(\mu)<\infty,
\qquad 1\leq j\leq k.
\]
Lemma~\ref{lem:pinned-star-induction} therefore gives a Borel set
\(E_\mu\subset\supp\mu\), with \(\mu(E_\mu)=1\), such that
\[
(\Phi_{k,x})_\#\mu^k\ll\cL^{t_k}
\qquad (x\in E_\mu).
\]

Let \(\Omega_{k,x}\) denote the corresponding pinned ordinary-distance
configuration measure. Then
\[
(\Phi_{k,x})_\#\mu^k
=
\operatorname{Sq}_\#\Omega_{k,x}.
\]
Since \(\mu\) is compactly supported, Lemma~\ref{lem:squaring-ac}
implies
\[
\Omega_{k,x}\ll\cL^{t_k}.
\]
The degenerate ordinary-distance tuples form a Lebesgue-null set, so
\[
\cL^{t_k}
\bigl(
\Delta^{\mathrm{nd}}_{k,x}(\supp\mu)
\bigr)>0.
\]
\end{proof}

\begin{remark}
\label{rem:pinned-measure-hypotheses}
Condition\eqref{eq:line-fourier-condition} holds, for example, if
\[
\lvert\wh\mu(\xi)\rvert
\lesssim
(1+\lvert\xi\rvert)^{-\beta/2}
\]
for some \(\beta>2\). The Fourier condition and the finite
\(k\)-energy condition in Corollary~\ref{cor:pinned-measure-criterion}
are conditions on the same measure. When \(k>2\), the separate
set level assumptions that the Fourier dimension of \(E\) is greater
than \(2\) and that \(\dim_{\mathrm H}(E)>k\) do not by themselves
guarantee a measure satisfying both hypotheses, and hence do not by
themselves allow an application of the corollary. When \(k=2\), by
contrast, Fourier dimension greater than \(2\) is sufficient, since
Fourier decay with exponent \(\beta>2\) also implies
\(I_2(\mu)<\infty\).
\end{remark}

\subsection{Proof of Theorem~\ref{thm:pinned-salem}}

The proof of Theorem~\ref{thm:pinned-salem} will be a consequence of
Corollary~\ref{cor:pinned-measure-criterion} once we verify its two hypotheses for a suitable Salem measure.

Now choose
\[
k<\sigma<\dim_H(E),
\]
and let $\mu$ be a probability measure supported on $E$ such that
\[
|\widehat{\mu}(\xi)|
\lesssim_\sigma
(1+|\xi|)^{-\sigma/2}.
\]
Since $k\ge 2$, we have $\sigma>2$. Therefore, for every
$\omega\in S^{d-1}$,
\[
\int_{\R}
|\widehat{\mu}(r\omega)|\,dr
\lesssim_\sigma
\int_{\R}
(1+|r|)^{-\sigma/2}\,dr
<\infty,
\]
uniformly in $\omega$.

We also claim that
\[
I_k(\mu)<\infty.
\]
Indeed, by the Fourier representation of the Riesz energy,
\[
I_k(\mu)
=
c_{d,k}
\int_{\R^d}
|\widehat{\mu}(\xi)|^2
|\xi|^{k-d}\,d\xi.
\]
Using the decay of $\widehat{\mu}$ and polar coordinates,
\[
I_k(\mu)
\lesssim
\int_0^1 r^{k-1}\,dr
+
\int_1^\infty r^{k-\sigma-1}\,dr
<\infty,
\]
since $k<\sigma$.

Thus $\mu$ satisfies both hypotheses of
Corollary\ref{cor:pinned-measure-criterion}. Hence there exists a Borel set
$E_\mu\subset\supp\mu$, with $\mu(E_\mu)=1$, such that for every
$x\in E_\mu$,
\[
(\Phi_{k,x})_\#\mu^k\ll\cL^{t_k},
\]
and the corresponding pinned ordinary-distance configuration measure
is also absolutely continuous. In particular,
\[
\cL^{t_k}
\bigl(
\Delta_{k,x}^{\mathrm{nd}}(\supp\mu)
\bigr)
>0.
\]
Since $\supp\mu\subset E$, it follows that
\[
\cL^{t_k}
\bigl(
\Delta_{k,x}^{\mathrm{nd}}(E)
\bigr)
>0
\qquad
(x\in E_\mu).
\]
This proves Theorem~\ref{thm:pinned-salem}.

We finish by comparing the Salem argument with the affine-span method of
Theorem~\ref{thm:euclidean-codim-one}.  The following remark identifies
the bottleneck that prevents the latter from reaching the threshold $k$.

\begin{remark}[]
\label{rem:affine-span-scope}
Let $2\le k\le d-2$, set $\ell=k-1$ and
$c=d-k+1$.  If one repeats the affine-span argument of
Sections~7--9 at rank $k$, then the Blaschke--Petkantschin step requires
\[
\int V_\ell(Y)^{-c}\,d\mu^k(Y)<\infty.
\]
For an $s$-Frostman measure, the worst Gram--Schmidt factor is the last
one, for which the preceding affine span has dimension at most $k-2$.
The Frostman tube estimate therefore gives summability when
\[
s-(k-2)>d-k+1,
\]
that is,
\[
s>d-1.
\]

By contrast, the corresponding cylindrical estimate is controlled by
$I_k(\mu)$, and hence only requires $s>k$.  Thus, within the present
affine-span implementation, the bottleneck is the absolute continuity of
the affine-span distribution rather than the distance-star estimate.
Consequently, rerunning the codimension-one argument directly at rank
$k$ does not improve the threshold $d-1$.

This is precisely what the Salem argument avoids: Fourier decay supplies
uniform control of one-dimensional projections and finite $I_k$-energy,
allowing the proof to get the threshold
\[
\dim_{\mathrm H}(E)>k.
\]
\end{remark}

\end{document}